%% file: CzechV-8.tex
\pdfoutput=1

\documentclass[
10pt, 
a4paper, 
oneside, 
headinclude,footinclude, 
BCOR5mm, 
]{scrartcl}

\usepackage{etex}

\input{structure.tex} 

\usepackage{amsmath,amsfonts,amscd}
\usepackage{amsthm}
\usepackage{amssymb}
\usepackage{eucal,upgreek}
\usepackage{yfonts,mathrsfs}
\usepackage{pb-diagram,pb-xy}
\usepackage[cmtip,arrow]{xy}

\usepackage{graphicx}
\usepackage{pdfpages}
\usepackage{stmaryrd}
\usepackage{graphics}
\usepackage{color,bm,stmaryrd}
\usepackage{pdfcomment}
\usepackage[toc,page]{appendix}

\usepackage{tikz}
\usepackage{braids}

\usetikzlibrary{decorations.pathreplacing}

\newcommand{\floor}[1]{{\left\lfloor{#1}\right\rfloor}}
\newcommand{\ceil}[1]{{\left\lceil{#1}\right\rceil}}
\newcommand{\disc}[0]{{\mathbb{D}^{2}}}

\newcommand{\plane}[0]{\mathbb{R}^2}
\newcommand{\rr}[0]{\mathbb{R}}
\newcommand{\zz}[0]{\mathbb{Z}}
\newcommand{\nn}[0]{\mathbb{N}}
\newcommand{\sbb}[0]{\mathbb{S}}
\newcommand{\Mod}[0]{{\rm Mod}}

\newcommand{\HH}{{\mathcal{H}}}

\newcommand{\balpha}{{\bm{\alpha}}}
\newcommand{\bbeta}{{\bm{\beta}}}
\newcommand{\bgamma}{{\bm{\gamma}}}

\newcommand{\ba}{{\bm{a}}}
\newcommand{\bb}{{\bm{b}}}

\newcommand{\bj}{{\bm{j}}}
\newcommand{\ii}{{\bm{i}}}
\newcommand{\h}{{\bm{h}}}
\newcommand{\hh}{{\mathrm{h}}}
\newcommand{\possim}{\stackrel{\mbox{\textnormal{\raisebox{0ex}[0ex][0ex]{$.$}}}}{\sim}}

\newcommand{\HHH}{{\mathrm{H}}}
\newcommand{\W}{{\mathscr{W}}}
\newcommand{\U}{{\mathscr{U}}}

\newcommand{\BB}{{\mathscr{B}}}
\newcommand{\BBB}{{\mathcal{B}}}

\newcommand{\RR}{{\mathcal{R}}}
\newcommand{\E}[0]{\mathbb{E}}
\newcommand{\A}[0]{\mathbb{A}}
\newcommand{\B}[0]{\mathbb{B}}
\newcommand{\T}[0]{\mathbb{T}}
\newcommand{\D}[0]{\mathbb{D}}
\newcommand{\C}[0]{\mathcal{C}}
\newcommand{\CC}[0]{\mathscr{C}}

\newcommand{\llb}{\llbracket}
\newcommand{\rrb}{\rrbracket}

\newcommand{\aset}[0]{\textsc{a}}
\newcommand{\bset}[0]{\textsc{b}}
\newcommand{\cset}[0]{\textsc{c}}
\newcommand{\Bd}[0]{\mathrm{B}}

\DeclareMathOperator{\ev}{{\mathcal{E}}}

\DeclareMathOperator{\Homeo}{Homeo}
\DeclareMathOperator{\Diff}{Diff}
\DeclareMathOperator{\Symp}{Symp}

\DeclareMathOperator{\Ham}{Ham}
\DeclareMathOperator{\Inv}{Inv}
\DeclareMathOperator{\rel}{\mathrm{rel}}
\DeclareMathOperator{\inter}{int}

\DeclareMathOperator{\id}{id}

\DeclareMathOperator{\Conf}{{\mathscr{D}}}

\DeclareMathAlphabet{\mathpzc}{OT1}{pzc}{m}{it}

\title{\normalfont\spacedallcaps{Symplectomorphisms and discrete braid invariants}} 

\author{\spacedlowsmallcaps{Aleksander Czechowski* \& Robert Vandervorst**}} 

\date{\small \today} 

\begin{document}

\theoremstyle{definition}
\newtheorem{thm}{Theorem}[section]
\newtheorem{cor}[thm]{Corollary}
\newtheorem{lem}[thm]{Lemma}

\theoremstyle{definition}
\newtheorem{prop}[thm]{Proposition}

\theoremstyle{definition}
\newtheorem{rem}[thm]{Remark}

\theoremstyle{definition}
\newtheorem{defn}[thm]{Definition}

\theoremstyle{definition}
\newtheorem{exm}[thm]{Example}


\renewcommand{\sectionmark}[1]{\markright{\spacedlowsmallcaps{#1}}} 
\lehead{\mbox{\llap{\small\thepage\kern1em\color{halfgray} \vline}\color{halfgray}\hspace{0.5em}\rightmark\hfil}} 

\pagestyle{scrheadings} 


\maketitle 

\setcounter{tocdepth}{2} 

\tableofcontents 




\section*{Abstract} 

Area and orientation preserving diffeomorphisms of the standard 2-disc, referred to as symplectomorphisms of $\disc$, allow  decompositions in terms of \emph{positive} twist diffeomorphisms.
Using the latter decomposition we utilize the Conley index theory of discrete braid classes  as introduced in \cite{BraidConleyIndex,GVV-pre} in order to obtain a
Morse type forcing theory of periodic points: a priori information about periodic points determines a mapping class which may force additional periodic points.


{\let\thefootnote\relax\footnotetext{* \textit{Institute of Computer Science and Computational Mathematics, Jagiellonian University, Krak\'ow}}}

{\let\thefootnote\relax\footnotetext{** \textit{Department of Mathematics, VU University, Amsterdam}}}


\newpage 


\begin{sloppypar}

\section{Prelude}
Let $\disc\subset\plane$ be the standard unit 2-disc with coordinates $z=(x,y)\in \plane$, and let $\omega=dx\wedge dy$ be the standard area 2-form on $\plane$.
A diffeomorphism $F\colon\disc \to\disc$ is said to be symplectic  if $F^*\omega = \omega$ --- area and orientation preserving ---
and is referred to as a \emph{symplectomorphism} of $\disc$.
Symplectomorphisms of the 2-disc form a group which is denoted by $\Symp(\disc)$. 
A diffeomorphism $F$ 
is \emph{Hamiltonian} if it is given as the time-1 mapping of a Hamiltonian system
\begin{equation}\label{HE} 
    \begin{aligned}
      \dot{x}&=\partial_y H(t,x,y); \\ \dot{y}&=-\partial_xH(t,x,y),
    \end{aligned}
\end{equation}
where $ H \in C^{\infty}(\rr\times \disc)$ a the Hamiltonian function with the additional property that $H(t, \cdot)|_{\partial \disc} = const.$ for all $ t \in \rr$. The set of  Hamiltonians satisfying these requirements is denoted by $\HH(\disc)$ and
the associated flow of \eqref{HE} is denoted by $\psi_{t,H}$. The group $\Ham(\disc)$ signifies the group of Hamiltonian diffeomorphisms of $\disc$.
 Hamiltonian diffeomorphisms are symplectic by construction. For the 2-disc these notions are equivalent, i.e. $\Symp(\disc) = \Ham(\disc)$ and we may therefore   study Hamiltonian systems in order to prove properties about symplectomorphisms of $\disc$, cf.\ \cite{Boyland2005}, and  Appendix \ref{sec:sympMCG}.

A subset $B\subset \disc$ is a invariant set for $F$ if $F(B) = B$.
 We are interested in finite invariant sets.
 Such invariant sets consist of periodic points, i.e.
points $z\in\disc$ such that $F^k(z) = z$ for some $k\ge 1$.
Since $\partial \disc$ is also invariant, periodic point are either in ${\rm int~}\disc$, or $\partial \disc$.

The main result of this paper concerns a
 \emph{forcing} problem. Given a finite invariant set $B\subset \inter \disc$ for $F \in \Symp(\disc)$, do there 
exist additional periodic points? More generally, does there exist a finite   invariant set $A\subset\inter \disc$, with $A\cap B=\varnothing$? This is much alike a similar question for the discrete dynamics on an interval, where 
the famous Sharkovskii theorem establishes a forcing order among periodic points based on their period.
In the 2-dimensional situation such an order is much harder to establish, cf.\ \cite{Boyland2005}.
The main results are based on braiding properties of periodic points and are stated and proved in Section \ref{sec:main1}.
The braid invariants introduced in this paper add additional information to existing invariants in the area-preserving case.
For instance,  Example \ref{exm:exist3} describes a braid class which forces additional invariant sets solely in the area-preserving case
hence extending the  non-symplectic methods described  in      \cite{JiangZheng}.
%

The theory in this paper can be further
genelarized to include symplectomorphisms of bounded subsets of $\rr^2$ with smooth boundary, eg. annuli, and  symplectomorphisms of $\rr^2$.

\vskip.3cm

{\bf Acknowledgements:} AC was supported by the Foundation for Polish Science under the MPD Programme `Geometry
and Topology in Physical Models', co-financed by the EU European Regional Development Fund,
Operational Program Innovative Economy 2007-2013.

\section{Mapping class groups}
\label{sec:discinv}
A priori knowledge of finite invariant sets $B$ for $F$
categorize mappings in so-called mapping classes.
%
%
Traditionally   mapping class groups are defined for orientation preserving homeomorphisms,
 cf.\  \cite{G3} for an overview. 
Denote by $\Homeo^+(\disc)$ the space  of orientation preserving homeomorphisms  and by $\Homeo^+_0(\disc)$ the homeomorphisms that leave the boundary point wise invariant.
Two homeomorphisms $F,G\in \Homeo^+(\disc)$ are   isotopic if there exists an isotopy $\phi_t$, with $\phi_t\in\Homeo^+(\disc)$ for all $t\in [0,1]$,
such that $\phi_0=F$ and $\phi_1 = G$.
The equivalence classes in $\pi_0\bigl(\Homeo^+(\disc)\bigr) = \Homeo^+(\disc)/\!\!\sim$ are called \emph{mapping classes} and  form a group under composition. The latter is referred to as the  \emph{mapping class group}  of the  2-disc and is denoted by $\Mod(\disc)$. 
For homeomorphisms that leave the boundary point wise invariant the mapping class  group is denoted by $\Mod_0(\disc) = \pi_0\bigl(\Homeo^+_0(\disc)\bigr)$. In  Appendix \ref{sec:MCGclbr} we provide proofs of the relevant facts about mapping class groups.
\begin{prop}
\label{prop:MCG11}
Both mapping class groups $\Mod(\disc)$ and $\Mod_0(\disc)$ are trivial.
\end{prop}
%
The   mapping class groups $\Mod(\disc)$ and $\Mod_0(\disc)$ may also be defined using diffeomorphisms, cf.\ Appendix \ref{sec:MCGclbr}.  
In  Proposition \ref{prop:mapclass1}, we show that $\pi_0\bigl(\Symp(\disc)\bigr) = \Mod(\disc)$ and in Propositon
  \ref{prop:mapclass3} we show that $\Ham(\disc) = \Symp(\disc)$, which implies that every homeomorphism, or diffeomorphism is isotopic to
  a Hamiltonian symplectomorphism.

\begin{prop}
\label{prop:MCG11a}
$\pi_0\bigl(\Symp(\disc)\bigr) = \pi_0\bigl(\Ham(\disc)\bigr) =\Mod(\disc)\cong 1$.
\end{prop}

More refined information about mapping classes is  obtained by considering finite invariant sets $B$.
This leads to the notion of the \emph{relative mapping classes}. 
Two homeomorphisms $F,G\in \Homeo^+(\disc)$ are of the same mapping class \emph{relative to} $B$ if there
exists an isotopy $\phi_t$,  with $\phi_t\in\Homeo^+(\disc)$ and $\phi_t(B) = B$ for all $t\in [0,1]$, such that
$\phi_0=F$ and $\phi_1=G$. The subgroup of such homeomorphisms is denoted by $\Homeo^+(\disc\rel B)$ and $\Homeo^+_0(\disc\rel B)$
in case $\partial\disc$ is point wise invariant.
The associated mapping class groups are denoted by 
$\Mod(\disc\rel B) = \pi_0\bigl(\Homeo^+(\disc\rel B) \bigr)$ and
$\Mod_0(\disc\rel B) = \pi_0\bigl(\Homeo^+_0(\disc\rel B) \bigr)$ respectively.

\begin{prop}
\label{prop:MCG12}
$\Mod(\disc\rel B)\cong \BB_m/Z(\BB_m)$ and $\Mod_0(\disc\rel B)\cong \BB_m$,
where $\BB_m$ is the Artin braid group, with $m = \# B$ and $Z(\BB_m)$ is the center of the braid group.
\end{prop}

Let $\C_m\disc$ be the \emph{configuration space} of unordered configurations of $m$ points in
$\disc$. 
\emph{Geometric braids}  on $m$ strands on $\disc$ are closed loops  in $\C_m\disc$ based at $B_0 = \{z_1,\cdots,z_m\}$, where the points $z_i$ are defined
as follows: $z_i = (x_i,0)$, $x_0=-1$, and $x_{i+1}= x_i + 2/(m+1)$.
The \emph{classical braid group} on $\disc$ is the fundamental group $\pi_1\bigl(\C_m\disc,B_0\bigr)$ and is denote by $\BBB_m\disc$.
The  (algebraic) \emph{Artin braid group} $\BB_{m}$ is a free group spanned by the $m-1$ generators $\sigma_{i}$, modulo
following relations:
\begin{align}\label{eqn:braidrel}
 \begin{cases}
  \sigma_{i} \sigma_{j} = \sigma_{j} \sigma_{i}, & \ |i-j| \geq 2,\ i,j \in \{1, \dots ,m-1\} \\
  \sigma_{i} \sigma_{i+1} \sigma_{i} = \sigma_{i+1} \sigma_{i} \sigma_{i+1}, & \ 1\le i \le m-2.
 \end{cases}
\end{align} 
Full twists are denoted algebraically by $\square=  (\sigma_{1} \dots \sigma_{m-1})^{m}$ and generate the center of the braid group $\BB_m$.
Presentation of words consisting only of the $\sigma_i$'s (not the inverses) and the relations in \eqref{eqn:braidrel} form a monoid
which is called the \emph{positive braid monoid} $\BB_m^+$.

There exists a canonical isomorphism $\ii_m\colon \BB_m \to \BBB_m\disc$, cf.\ \cite[Sect.\ 1.4]{Birman}.
For closed loops $\bbeta(t)$ based at $B\in \C_m\disc$ we have a canonical isomorphism   $\bj_B\colon\pi_1\bigl(\C_m\disc,B\bigr)\to \pi_1\bigl(\C_m\disc,B_0\bigr) =\BBB_m\disc$. 
Let $p\colon [0,1]\to \C_m\disc$ be a path connecting $B_0$ to $B$, then define $\bj_B\bigl([\bbeta]_B\bigr) := [(p\cdot \bbeta)\cdot p^*]_{B_0}
= [p\cdot(\bbeta\cdot p^*)]_{B_0}$, where $p^*$ is the inverse path connecting $B$ to $B_0$. The definition of $\bj_B$ is independent of the chosen path $p$.
This yields the isomorphism 
$
\imath_B = \ii_m^{-1}\circ \bj_B\colon \pi_1\bigl(\C_m\disc,B\bigr) \to
\BB_m.
$

The construction of the isomorphism $\Mod_0(\disc\rel B) \cong \BBB_{m}\disc$ can be  understood as follows, cf.\ \cite{Birman}, \cite{Birman2}. For 
$F\in \Homeo^+_0(\disc\rel B)$
choose an   isotopy $\phi_t\in \Homeo^+_0(\disc), \ t \in [0,1]$, 
such that $\phi_1=F$.
Such an isotopy  exists since $\Homeo^+_0(\disc)$ is contractible, cf.\ Propostion \ref{prop:MCG11}.
For $G\in [F]\in \Mod_0(\disc\rel B)$, the composition and scaling  of the isotopies defines isotopic braids based at $B\in \C_m\disc$.
The isomorphism $\jmath_B\colon \Mod_0(\disc\rel B) \to \BB_m$  is given by $\jmath_B([F]) = \iota_B\bigl(d_*^{-1}([F])\bigr)= \imath_B([\bbeta]_B)$, with $\bbeta(t) = \phi_t(B)$
the geometric braid generated by $\phi_t$. The isomorphism $d_*$ is given in Appendix \ref{subsec:braidMCG} and $[\bbeta]_B$ denotes the homotopy class in $\pi_1\bigl(\C_m\disc,B\bigr)$.
For $\Mod(\disc\rel B)$ we use the same notation for the isomorphism which is given by
\[
\jmath_B\colon\Mod(\disc\rel B) \cong \BB_{m}/Z(\BB_m),\quad [F] \mapsto \jmath_B([F])= \beta\!\!\!\!\mod\square,
\]
where $\beta = \imath_B\bigl( [\bbeta]_B\bigr)$.
%
The above mapping class groups can also be defined using diffeomorphisms and  symplectomorphisms. 

\begin{prop}
\label{prop:MCG12a} 
 $ \pi_0\bigl(\Ham(\disc\rel B)\bigr) =\Mod(\disc\rel B)\cong \BB_m/Z(\BB_m)$.
\end{prop}
In Appendix \ref{sec:sympMCG} we show that $\pi_0\bigl(\Symp(\disc\rel B)\bigr) = \Mod(\disc\rel B)$ and  that $\Symp(\disc\rel B) = \Ham(\disc\rel B)$
and therefore that every mapping class can be represented by Hamiltonian symplectomorphisms.


\section{Braid classes}
\label{subsec:2color}
Considering free loops in a configuration space as opposed to based loops leads to classes of closed braids, which are the key tool for studying periodic points.

\subsection{Discretized braids}
\label{subsec:discbr}
From \cite{BraidConleyIndex} we recall the notion of positive piecewise linear braid diagrams and discretized braids.
\begin{defn}
\label{PL}
The space of {\em discretized period $d$ closed braids on $n$ strands},
denoted $\Conf^d_m$, is the space of all pairs $(\bb,\tau)$ where
$\tau\in S_m$ is a permutation on $m$ elements, and $\bb$ is an 
unordered set of $m$ {\em strands}, $\bb=\{\bb^\mu\}_{\mu=1}^m$,
defined as follows:
\begin{enumerate}
\item[(a)]
	each strand  
	$\bb^\mu=(x^\mu_0,x^\mu_1,\ldots,x^\mu_d)\in\rr^{d+1}$
	consists   of $d+1$ {\em anchor points} $x_j^\mu$;
\item[(b)] $x^\mu_d = x^{\tau(\mu)}_0$
	for all $\mu=1,\ldots,m$;
\item[(c)]
	for any pair of distinct strands $\bb^\mu$ and $\bb^{\mu'}$
	such that $x^\mu_j=x^{\mu'}_j$ for some $j$,
	the \emph{ transversality} condition
	 $\bigl(x^\mu_{j-1}-x^{\mu'}_{j-1}\bigr)
	\bigl(x^\mu_{j+1}-x^{\mu'}_{j+1}\bigr) < 0$ holds.
\end{enumerate}
%
\end{defn}

\begin{rem}
Two discrete braids $(\bb,\tau)$ and $( \bb', \tau')$ are close if the strands $\bb^{\zeta(\mu)}$ and $ \bb'^\mu$
are close in $\rr^{md}$ for some permutation $\zeta$ such that $ \tau'=\zeta\tau\zeta^{-1}$.
We suppress the permutation $\tau$ from the notation. 
Presentations via the braid monoid $\BB_m^+$ store the permutations.
\end{rem}

\begin{defn}[cf.\ \cite{BraidConleyIndex}]
\label{defn:closure}
The closure $\bar\Conf_m^d$ of the space $\Conf_m^d$ consists of pairs $(\bb,\tau)$ for which (a)-(b) in Definition \ref{PL} are satisfied.
\end{defn}

The path components of $\Conf_m^d$ are the \emph{discretized braids classes} $[\bb]$.
Being in the same path connected component is an equivalence relation on $\Conf_m^d$, where the braid classes are the
equivalence classes expressed by the notation $\bb,\bb'\in [\bb]$, and $\bb\sim\bb'$. 
The associated permutations $\tau$ and $\tau'$ are conjugate. A path connecting $\bb$ and $\bb'$ is called a \emph{positive isotopy} and the equivalence relation is   referred to \emph{positively isotopic}.

To a configuration $\bb\in\Conf_m^d$  one can associate a 
\emph{piecewise linear braid diagram} $\Bd(\bb)$. For 
each strand $\bb^\mu\in \bb$, consider the piecewise-linear (PL) 
interpolation
\begin{equation}\label{interpolate1}
\Bd^{\mu}(t) := x^\mu_{\floor{d\cdot t}}+(d\cdot t-\floor{d\cdot t})
	(x^\mu_{\ceil{d\cdot t}}-x^\mu_{\floor{d\cdot t}}),
\end{equation}
for $t\in[0,1]$.
The braid diagram $\Bd(\bb)$ is then defined to be the 
superimposed graphs of all the functions $\Bd^{\mu}(t)$.
A braid diagram $\Bd(\bb)$ is not only a good bookkeeping tool for keeping track of the strands
in $\Bd(\bb)$, but also plays natural the role of a braid diagram projection with only positive intersections, cf.\ Section \ref{subsec:discbrinv12}.

The set of $t$-coordinates of intersection points in $\Bd(\bb)$ is denoted by $\{t_i\}$, $i=1,\cdots,|\bb|$, where $|\bb|$ is the total number of
intersections in $\Bd(\bb)$ counted with multiplicity. The latter is also referred to as the \emph{word metric} and is an invariant for $\bb$.
A discrete braid $\bb$ is \emph{regular} if all points $t_i$ and anchor points $x_j^\mu$ are distinct.
The regular discrete braids in $[\bb]$ form a dense subset and every discrete braid is positively isotopic to a regular discrete braid. 
To a regular discrete braid $\bb$ one can assign a unique positive word $\beta = \beta(\bb)$ defined as follows:
\begin{equation}
\label{eqn:word1}
\bb \mapsto \beta(\bb) = \sigma_{k_1} \cdots \sigma_{k_\ell},
\end{equation}
 where $k_i$ and $k_i +1$ are the positions that intersect at $t_i$, cf.\ \cite[Def.\ 1.13]{Dehornoy1}. 
On the positive braid monoid $\BB_m^+$ two positive words $\beta$ and $\beta'$ are positively equal,
notation $\beta \doteq\beta'$, if they represent the same element in $\BB_m^+$ using  the relations in \eqref{eqn:braidrel}.
On $\BB_m^+$ we define an equivalence relation which acts as an analogue of conjugacy in the braid group, cf.\ \cite[Sect.\ 2.2]{BDV}.
For a given word $\sigma_{i_1}\cdots\sigma_{i_n}$, define the relation
\[
\sigma_{i_1}\sigma_{i_2}\cdots\sigma_{i_n} \equiv \sigma_{i_2}\cdots\sigma_{i_n}\sigma_{i_1}.
\]

\begin{defn}
\label{defn:equiv12}
Two positive words $\beta,\beta'\in \BB_m^+$ are \emph{positively conjugate}, notation
$\beta \possim \beta'$, if there exists a sequence of words $\beta_0,\cdots,\beta_\ell\in \BB_m^+$, with $\beta_0=\beta$ and $\beta_\ell =\beta'$, such
that for all $k$, either  $\beta_k\doteq \beta_{k+1}$, or $\beta_k\equiv \beta_{k+1}$ 
\end{defn}

Positive conjugacy is an equivalence relation on $\BB_m^+$ and the set of  positive conjugacy classes $\llb\beta\rrb$ of the braid monoid $\BB_m^+$ is  denoted by
$\CC \BB_m^+$. 

The above defined assignment $\bb \mapsto \beta(\bb)$ can be extended to all discrete braids. A discrete braid $\bb$ is positively isotopic to a regular braid $\bb'$ and the mapping $\Conf_m^d \to \CC\BB_m^+$, given by $\bb \mapsto \llb\beta(\bb)\rrb$, is well-defined
by choosing $\beta(\bb)$ to be any representative in the positive conjugacy class $\llb\beta(\bb')\rrb$.
Observe that for fixed $d$ the mapping $\Conf_m^d \to \CC\BB_m^+$ is not surjective.

\begin{rem}
The positive conjugacy relation defined in Definition \ref{defn:equiv12} is symmetric by construction since it is defined on finite words.
For instance, consider $\sigma_1\sigma_2\sigma_3 \equiv \sigma_2\sigma_3\sigma_1$.
The question whether the $\sigma_2\sigma_3\sigma_1 \equiv \sigma_1\sigma_2\sigma_3$ is answered as follows:
$\sigma_2\sigma_3\sigma_1 \equiv \sigma_3\sigma_1\sigma_2 \equiv \sigma_1\sigma_2\sigma_3$, which, by Definition \ref{defn:equiv12},
shows that $\sigma_2\sigma_3\sigma_1 \equiv \sigma_1\sigma_2\sigma_3$.
\end{rem}

The presentation of discrete braids via words in $\BB_m^+$ yields the 
 following alternative equivalence relation.
\begin{defn}
\label{defn:topeq}
Two discretized braids $\bb, \bb'\in \Conf_m^d$ are \emph{topologically equivalent} if $\beta(\bb) \possim \beta(\bb')$ in $\BB_m^+$, i.e. 
$\beta(\bb)$ and $\beta(\bb')$ are positively conjugate.
Notation: $\bb \possim \bb'$.
\end{defn}

Summarizing, $\bb\sim\bb'$ implies $\bb \possim \bb'$
and $\possim$ defines a coarser equivalence relation on $\Conf_m^d$.
The equivalence classes with respect to $\possim$ are denote by $[\bb]_{\possim}$.
The converse   is not true in general, cf.\ \cite[Fig.\ 8]{BraidConleyIndex}.
Following \cite[Def.\ 17]{BraidConleyIndex}, a discretized braid class $[\bb]$ is \emph{free} if
%
$[\bb] = [\bb]_{\possim}$.
%
\begin{prop}[\cite{BraidConleyIndex}, Prop.\ 27]
\label{prop:free}
If $d>|\bb|$, then $[\bb]$ is a free braid class.
\end{prop}

Taking $d$ sufficiently large is a sufficient condition to ensure free braid classes, but this condition is not a necessary condition.

\begin{figure}[tb]
\centering
\begin{tikzpicture}[xscale=0.4, yscale=0.4, line width = 1.5pt] 
  \draw[-] (0,0) -- (2,6) -- (4,0); 
  \draw[fill] (0,0) circle (0.07)
              (2,6) circle (0.07)
              (4,0) circle (0.07);

   \draw[-] (0,2) -- (2,2) -- (4,2);            
   \draw[fill] (0,2) circle (0.07)
               (2,2) circle (0.07)
               (4,2) circle (0.07);

   \draw[-] (0,4) -- (2,4) -- (4,4);
   \draw[fill] (0,4) circle (0.07)
               (2,4) circle (0.07)
               (4,4) circle (0.07);
\foreach \x in {0,4}
 \draw[line width=1.0pt, -] (\x,-1) -- (\x,7); 

\end{tikzpicture}
\qquad
\qquad
\begin{tikzpicture}[xscale=0.4, yscale=0.4, line width = 1.5pt]
  \draw[-] (0,6) -- (2,0) -- (4,6); 
  \draw[fill] (0,6) circle (0.07)
              (2,0) circle (0.07)
              (4,6) circle (0.07);

   \draw[-] (0,2) -- (2,2) -- (4,2);            
   \draw[fill] (0,2) circle (0.07)
               (2,2) circle (0.07)
               (4,2) circle (0.07);

   \draw[-] (0,4) -- (2,4) -- (4,4);
   \draw[fill] (0,4) circle (0.07)
               (2,4) circle (0.07)
               (4,4) circle (0.07);
\foreach \x in {0,4} 
 \draw[line width=1.0pt, -] (\x,-1) -- (\x,7); 

\end{tikzpicture}
\qquad
\qquad
\begin{tikzpicture}[xscale=0.4, yscale=0.4, line width = 1.5pt]
  \draw[-] (0,0) -- (2,6) -- (4,3.5) -- (6,0); 
  \draw[fill] (0,0) circle (0.07)
              (2,6) circle (0.07)
              (4,3.5) circle (0.07)
              (6,0) circle (0.07);

   \draw[-] (0,2) -- (2,2) -- (4,2) -- (6,2);            
   \draw[fill] (0,2) circle (0.07)
               (2,2) circle (0.07)
               (4,2) circle (0.07)
               (6,2) circle (0.07);

   \draw[-] (0,4) -- (2,4) -- (4,4) -- (6,4);
   \draw[fill] (0,4) circle (0.07)
               (2,4) circle (0.07)
               (4,4) circle (0.07)
               (6,4) circle (0.07);
\foreach \x in {0,6}
 \draw[line width=1.0pt, -] (\x,-1) -- (\x,7); 
\end{tikzpicture}
\caption{The left and middle diagrams show  representatives  $\bb,\bb'\in \Conf_3^2$ in Example \ref{exm:free2}.
The right diagram shows a representative the same topological braid class in $\Conf_3^3$ (free).}
\label{fig:free1}
\end{figure}
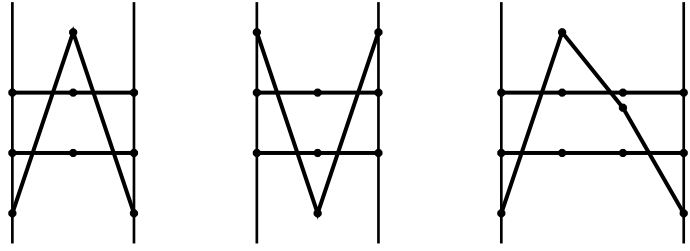

\begin{exm}\label{exm:free2}
Given the braids $\bb\in \Conf_3^2$ with  $\bb^1 = (1,4,1)$, $\bb^2 = (2,2,2)$ and $\bb^3 = (3,3,3)$ and consider the braid
class $[\bb]$, see Figure \ref{fig:free1}[left and middle]. Since $\bb$ is regular, $\beta(\bb)$ is uniquely defined and $\beta(\bb) = \sigma_1\sigma_2^2\sigma_1$.
Also define $\bb'\in \Conf_3^2$ with $\bb'^1 = (4,1,4)$, $\bb'^2=\bb^2$ and $\bb'^3=\bb^3$ and the braid
class $[\bb']$. Since $\bb'$ is also regular we have the unique braid word $\beta(\bb') = \sigma_2\sigma_1^2\sigma_2$.
Observe that $ \sigma_1\sigma_2^2\sigma_1\possim \sigma_2\sigma_1^2\sigma_2$, which implies that $\bb$ and
$\bb'$ are topologically equivalent. However, $\bb$ and $\bb'$ are not positively isotopic in $\Conf_3^2$ and $[\bb]$ and
$[\bb']$ are two different path components of $\Conf_3^2$.
The positive conjugacy class of $\sigma_1\sigma_2^2\sigma_1$ is given by $\llb\sigma_1\sigma_2^2\sigma_1\rrb = \{\sigma_1\sigma_2^2\sigma_1,\sigma_2^2\sigma_1^2,
\sigma_2\sigma_1^2\sigma_2,\sigma_1^2\sigma_2^2\}$. The words $\sigma_2^2\sigma_1^2$ and $\sigma_1^2\sigma_1^2$
are not represented in $\Conf_3^2$.
If we consider $ \bb''\in \Conf_3^3$ given by $\bb'' = \bigl\{ (1,4,1,1), (2,2,2,2),  (3,3,3,3)\bigr\}$, then
the associated braid class $[\bb'']$ is free, which confirms that the condition in Proposition \ref{prop:free} is not a necessary
condition, see Figure \ref{fig:free1}[right].
\end{exm}

 Let 
$\beta = \sigma_{i_1} \cdots \sigma_{i_d}\in \BB_m^+$ be a positive braid word, then define
\[
\ev_q(\beta) :=  \bb = \{ \bb^\mu\} \in \Conf_m^{d+q},\quad  \bb^\mu = (x_j^\mu),\quad\mu=1,\cdots, m,~~q\ge 0,
\]
with $x_0^\mu = \mu$, $x_j^\mu = x_0^{\sigma_{i_1} \cdots \sigma_{i_j}(\mu)}$, $j=1,\cdots,d$, and $x_{d+q}^\mu =\cdots = x_d^\mu$. 
The expression $\sigma_{i_1} \cdots \sigma_{i_j}(\mu)$, $\mu = 1,\cdots,m$ describes the permutation of the set $\{1,\cdots,m\}$, where
$\sigma_{i_1} \cdots \sigma_{i_j}$ is regarded as a concatenation of permutations given by the generators $\sigma_i$ interpreted as
a basic permutation of $i$ and $i+1$.
By Proposition \ref{prop:free}, $[\ev_q(\beta)]$ is free for all $q\ge 1$, and every
 $\llb\beta\rrb\in \CC\BB_m^+$  defines a free discrete braid class $[\ev_q(\beta)]$ in $\Conf_m^{d+q}$
 for all $q\ge 1$. 

\subsection{Discrete 2-colored braid classes}
\label{subsec:discbr2}
On closed configuration spaces we define the following product:
\[ 
\bar\Conf_n^d\times\bar\Conf_m^d \to \bar\Conf_{n+m}^d,\quad (\ba,\bb) \mapsto  \ba\sqcup\bb,
\]
where $\ba\sqcup\bb$ is the disjoint union of the strands in $\ba$ and $\bb$ regarded as an element in $\bar\Conf_{n+m}^d$.
The definition yields a canonical permutation on the labels in $\ba\sqcup\bb$.
Define the space of \emph{2-colored discretized braids} as the space of ordered pairs
\begin{equation}
\label{eqn:2color}
\Conf_{n,m}^d := \bigl\{ \ba\rel\bb:=(\ba,\bb)~|~ \ba\sqcup\bb \in \Conf_{n+m}^d\bigr\}.
\end{equation}
The strand labels in $\ba$ range from $\mu=1,\cdots,n$ and the strand labels in $\bb$ range from $\mu=n+1,\cdots,n+m$.
The associated permutation $\tau_{\ba,\bb} = \tau_\ba \oplus\tau_\bb \in S_{n+m}$, where
 $\tau_\ba\in S_n$ and $\tau_\bb \in S_m$, and $\tau_\ba$ acts on the labels
$\{1,\cdots,n\}$ and $\tau_\bb$ acts on the labels $\{{n+1},\cdots,{n+m}\}$.
The strands $\ba = \{x_j^{\mu}\}$, $\mu=1,\cdots,n$ are the \emph{red}, or \emph{free} strands and
the strands $\bb = \{x_j^{\mu}\}$, $\mu=n+1,\cdots,n+m$ are the \emph{black}, or \emph{skeletal} strands.
A path component $[\ba\rel\bb]$ in $\Conf_{n,m}^d$ is called a \emph{2-colored discretized braid class}.
The canonical projections are given by
$\varpi\colon\Conf_{n,m}^d \to \Conf_m^d$ with $\ba\rel\bb \mapsto \bb$ and by
$\varpi^*\colon\Conf_{n,m}^d \to \Conf_n^d$ with $\ba\rel\bb \mapsto \ba$.
The mapping $\varpi$ yields a fibration
\begin{equation}
\label{eqn:fiberbundle2}
[\ba]\rel \bb \to [\ba\rel \bb] \to [\bb].
\end{equation}
 The pre-images
$\varpi^{-1}(\bb) = [\ba]\rel \bb\subset  \Conf_{n}^d$,   are called the \emph{relative discretized braid class fibers}.

There exists a natural embedding $\Conf_{n,m}^d \hookrightarrow \Conf_{n+m}^d$, defined by $\ba\rel\bb \mapsto \ba\sqcup\bb$.
Via the embedding we   define the notion of topological equivalence of two 2-colored discretized braids:
$\ba\rel\bb \possim \ba'\rel\bb'$ if $\ba\sqcup\bb \possim \ba'\sqcup\bb'$. The associated equivalence classes are denoted by $[\ba\rel\bb]_{\possim}$, which are
not necessarily connected sets in $\Conf_{n,m}^d$. A 2-colored discretized braid class $[\ba\rel\bb]$ is free if $[\ba\rel\bb] = [\ba\rel\bb]_{\possim}$.
If $d>|\ba\sqcup\bb|$, then  $[\ba\rel\bb]$ is free by Proposition \ref{prop:free}. 

The set of collapsed singular braids in $\bar \Conf_{n,m}^d$ is given by:
\[
\begin{aligned}
\Sigma^- := \{\ba\rel\bb\in \bar\Conf_{n,m}^d~|~ \ba^\mu =~&\ba^{\mu'}, \hbox{or~} \ba^\mu=\bb^{\mu'}\\
&\hbox{for some~~}\mu\not=\mu',~\hbox{and~} \bb\in \Conf_m^d\}.
\end{aligned}
\]
A 2-colored discretized braid class $[\ba\rel\bb]$ is \emph{proper} if 
$\partial [\ba\rel\bb]_{\possim} \cap \Sigma^- = \varnothing$.
If a braid class $[\ba\rel\bb]$ is not proper it is called \emph{improper}.
In \cite{BDV} properness is considered in a more general setting. The notion of properness in this paper coincides with weak properness in \cite{BDV}.
%

%
A 2-colored discretized braid class $[\ba\rel\bb]$
is called \emph{bounded} if its fibers are bounded as  sets in $\rr^{nd}$. 
Note that $[\ba\rel\bb]$
is \emph{not} a bounded set in $\rr^{(n+m)d}$.


\subsection{Algebraic presentations}
\label{subsec:algpres}
Discretized braid classes are presented via the positive conjugacy classes of the positive braid monoid $\BB_m^+$.
For 2-colored discretized braids we seek a similar presentation.

%
In order to keep track of colors we define coloring on words in $\BB_{n+m}^+$.
Words in $\BB_{n+m}^+$ define associated permutations $\tau$ and the permutations $\tau$ yield partitions of the set $\{1,\cdots,n+m\}$.
Let $\gamma \in \BB_{n+m}^+$ be a word for which the induced partition   contains a union of equivalence classes $\aset \subset \{1,\cdots,n+m\}$
consisting of $n$ elements. The set $\aset$ is the \emph{red coloring} of length $n$ and the remaining partitions are colored black, denoted by $\bset$. 
The pair $(\gamma,\aset)$ is 
called a 2-colored positive braid word, see Figure \ref{fig:relative1}.
For a given coloring $\aset \subset \{1,\cdots,n+m\}$ of length $n$ the set of all words $(\gamma,\aset)$ forms a monoid which is denoted by $\BB^+_{n,m,\aset}$ and is referred as
the \emph{2-colored braid monoid} with coloring $\aset$.

Two pairs $(\gamma,\aset)$ and $(\gamma',\aset')$ are positively conjugate if $\gamma\possim \gamma'$ and
$\aset' = \zeta^{-1}(\aset)$, where $\zeta$ is a permutation conjugating the induced permutations $\tau_\gamma$ and $\tau_{\gamma'}$, i.e.  $\tau_{\gamma'} = \zeta\tau_\gamma\zeta^{-1}$.
If $\xi$ is another permutation such that $\tau_{\gamma'} = \xi\tau_\gamma\xi^{-1}$, then 
$ \zeta\tau_\gamma\zeta^{-1} =  \xi\tau_\gamma\xi^{-1}$. This implies that $\tau_\gamma =  \zeta^{-1}\xi\tau_\gamma\xi^{-1}\zeta$
and thus $\xi^{-1}\zeta(\aset) = \aset$, which is equivalent to 
$\zeta^{-1}(\aset) = \xi^{-1}(\aset)$. This shows
that the conjugacy relation in well-defined.
Positive conjugacy for 2-colored braid  words is again denoted by $(\gamma,\aset) \possim (\gamma', \aset')$
and a conjugacy class is denoted by $\llb\gamma,\aset\rrb$.
The set of 2-colored positive conjugacy classes 
with red colorings of length $n$ is denoted by $\CC\BB_{n,m}^+$.

The words corresponding   to the different colors $(\gamma,\aset)$  can be derived from the information in $(\gamma,\aset)$.
Let $\aset_0\subset \aset$ be a cycle of length $\ell\le n$ and let $k  \in \aset_0$. If $\gamma = \sigma_{i_1}\cdots \sigma_{i_d}$, then we define
an $\ell$-periodic sequence $\{k_j\}$, with 
\[
k_0 = k,\quad\hbox{and}\quad k_{j} = \sigma_{i_j}(k_{j-1}), ~~~j=1,\cdots,\ell d,
\]
by considering the word $\gamma^\ell$. Now use the following rule: if $k_j-k_{j-1} \not = 0$, remove 
$\sigma_{i_{j'}}$ from $\gamma$, for $j=1,\cdots, \ell d$, where $j' =j \!\!\mod d \in \{1,\cdots,d\}$.
Moreover, $\sigma_{i_j}$ is replaced by $\sigma_{i_j-1}$, if $k_j=k_{j-1}<i_j$ and $\sigma_{i_j}$ remains unchanged otherwise.
We repeat this procedure for all cycles in $\aset$ and we obtain the mapping $(\gamma,\aset) \mapsto \beta \in \BB_m^+$ denoted by
\[
\pi_{\aset}\colon \BB^+_{n,m,\aset} \to \BB^+_m.
\]
By considering the complementary color $\bset$ we construct a mapping $(\gamma,\bset) \mapsto \alpha\in \BB_n^+$ using the same scheme.

With the notion of coloring braid words we can encode the information of a 2-colored discretized braid $\ba\rel\bb$ in a 2-colored word
$(\gamma,\aset)$. Given $\ba\rel\bb\in \Conf_{n,m}^d$ (regular), we define 
\begin{equation}
\label{eqn:colored1}
\ba\rel\bb \mapsto  \gamma(\ba\rel\bb) := \beta(\ba\sqcup\bb),
\end{equation}
 cf.\ \eqref{eqn:word1} and the coloring $\aset = \cset^{-1}(\{1,\cdots,n\})$, where the permutation $\cset$
is defined as follows. Order the coordinates $x_0^{\mu_1} < \cdots < x_0^{\mu_{n+m}}$ and define
\[
\cset^{-1} = \left(\begin{array}{cccc}1 & 2 & \cdots & n+m \\\mu_1 & \mu_2 & \cdots & \mu_{n+m}\end{array}\right).
\]
The permutations $\tau_\gamma$ and $\tau_{\ba,\bb}$ are conjugated: $\tau_\gamma = \cset \tau_{\ba,\bb} \cset^{-1}$.
The mapping $\ba\rel\bb \mapsto (\gamma,\aset)$ is well-behaved under positive conjugacy: 
$\ba\rel\bb \possim \ba'\rel\bb'$ implies $(\gamma,\aset) \possim (\gamma',\aset')$.

Conversely, every positive conjugacy class $\llb\gamma,\aset\rrb$ determines a 2-colored discretized braid class $[\ba\rel\bb]_{\possim}$ via
the mapping $\gamma \mapsto \ev_q(\gamma) \in \Conf_{n+m}^{d+q}$, with $d$ the number of generators in $\gamma$ and $q\ge 0$, cf.\ Figure \ref{fig:relative1}[right]. 
The representation  $\ba \rel \bb\in \Conf_{n,m}^{d+q}$ is obtained from the coloring $\aset$.

\begin{defn}
\label{defn:proper2}
A positive conjugacy class $\llb\gamma,\aset\rrb$ is called proper if the associated discretized braid class  $[\ba\rel\bb]_{\possim}$  is  proper, cf.\ Section \ref{subsec:discbr2}.
\end{defn}


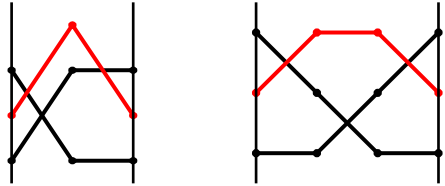
\begin{figure}[htb]
\centering
\begin{tikzpicture}[xscale=0.4, yscale=0.3, line width = 1.5pt]
   \draw[-] (0,0) -- (2,4) -- (4,4);            
   \draw[fill] (0,0) circle (0.07)
               (2,4) circle (0.07)
               (4,4) circle (0.07);

   \draw[-] (0,4) -- (2,0) -- (4,0);            
   \draw[fill] (0,4) circle (0.07)
               (2,0) circle (0.07)
               (4,0) circle (0.07);

   \draw[-,color=red] (0,2) -- (2,6) -- (4,2);            
   \draw[fill, color=red] (0,2) circle (0.07)
                           (2,6) circle (0.07)
                           (4,2) circle (0.07);

 \foreach \x in {0,4}
 \draw[line width=1.0pt, -] (\x,-1) -- (\x,7); 
\end{tikzpicture}
\qquad
\qquad
\begin{tikzpicture}[xscale=0.4, yscale=0.4, line width = 1.5pt]
   \draw[-] (0,0) -- (2,0) -- (4,2) -- (6,4);            
   \draw[fill] (0,0) circle (0.07)
               (2,0) circle (0.07)
               (4,2) circle (0.07)
               (6,4) circle (0.07);

   \draw[-] (0,4) -- (2,2) -- (4,0) -- (6,0);            
   \draw[fill] (0,4) circle (0.07)
               (2,2) circle (0.07)
               (4,0) circle (0.07)
               (6,0) circle (0.07);

   \draw[-,color=red] (0,2) -- (2,4) -- (4,4) -- (6,2);            
   \draw[fill, color=red]  (0,2) circle (0.07)
                           (2,4) circle (0.07)
                           (4,4) circle (0.07)
                           (6,2) circle (0.07);

 \foreach \x in {0,6}
 \draw[line width=1.0pt, -] (\x,-1) -- (\x,5); 
\end{tikzpicture}
\caption{Representations of the relative braid class $[\ba\rel\bb]$.}
\label{fig:relative1}
\end{figure}

\begin{exm}
Consider $\ba\rel\bb\in \Conf_{1,2}^2$, with  strands $\ba = \{(2,4,2)\}$, $\bb = \{(1,3,3), (3,1,1)\}$.
Since the strands in $\bb$ have labels $\mu=2,3$,
the permutation is given by $\tau_{\ba,\bb} = (23)$ and $\gamma = \beta(\ba\sqcup \bb) = \sigma_2\sigma_1\sigma_2$.
The coloring permutation is given as follows: $x_0^2<x_0^1<x_0^3$ and therefore
$\cset^{-1} = (12)$. 
The red coloring is given by $\aset = \cset(\{1\}) = \{2\}$.
We verify that $\tau_\gamma = \cset \tau_{\ba,\bb} \cset^{-1} = (2)(13) = (13)$, see Figure \ref{fig:relative1}[left].
The topological type of $\ba\rel \bb$ is given by the $\llb\sigma_2\sigma_1\sigma_2,\{2\}\rrb$ and $[\ba\rel\bb]_{\possim}$ is a proper and free 2-colored discrete braid class, see Figure \ref{fig:relative1}[right].
In order to compute the skeletal braid word $\beta$ we consider the following sequence:
\[
k_0 =2,\quad k_1 = \sigma_2(2) = 3,\quad k_2=\sigma_1(3) = 3,\quad k_3 = \sigma_2(3) =2.
\]
This yields the differences $k_1-k_0= 1$, $k_2-k_1 = 0$ and $k_3-k_2 = -1$, and therefore both letters $\sigma_2$ are removed from $\gamma$,
which gives $\gamma \mapsto \beta = \sigma_1$.
\begin{figure}[hbt]
\centering
\begin{tikzpicture}[xscale=0.4, yscale=0.25, line width = 1.5pt]
   \draw[-] (0,0) -- (2,4) -- (4,4);            
   \draw[fill] (0,0) circle (0.07)
               (2,4) circle (0.07)
               (4,4) circle (0.07);

   \draw[-] (0,4) -- (2,0) -- (4,0);            
   \draw[fill] (0,4) circle (0.07)
               (2,0) circle (0.07)
               (4,0) circle (0.07);

   \draw[-,color=red] (0,2) -- (2,6) -- (4,2);            
   \draw[fill, color=red] (0,2) circle (0.07)
                           (2,6) circle (0.07)
                           (4,2) circle (0.07);

   \draw[-] (0+4,0) -- (2+4,4) -- (4+4,4);            
   \draw[fill] (0+4,0) circle (0.07)
               (2+4,4) circle (0.07)
               (4+4,4) circle (0.07);

   \draw[-] (0+4,4) -- (2+4,0) -- (4+4,0);            
   \draw[fill] (0+4,4) circle (0.07)
               (2+4,0) circle (0.07)
               (4+4,0) circle (0.07);

   \draw[-,color=red] (0+4,2) -- (2+4,6) -- (4+4,2);            
   \draw[fill, color=red] (0+4,2) circle (0.07)
                           (2+4,6) circle (0.07)
                           (4+4,2) circle (0.07);

\foreach \x in {0,4,8}
 \draw[line width=1.0pt, -] (\x,-1) -- (\x,7); 

\foreach \x in {2,6}
 \draw[loosely dashed, line width=1.0pt, -] (\x,-1) -- (\x,7); 
\end{tikzpicture}
\qquad
\begin{tikzpicture}[xscale=0.4, yscale=0.25, line width = 1.5pt]
   \draw[-] (0,0) -- (2,4) -- (4,4);            
   \draw[fill] (0,0) circle (0.07)
               (2,4) circle (0.07)
               (4,4) circle (0.07);

   \draw[-] (0,4) -- (2,0) -- (4,0);            
   \draw[fill] (0,4) circle (0.07)
               (2,0) circle (0.07)
               (4,0) circle (0.07);

   \draw[-,color=red] (0,2) -- (2,3) -- (4,2);            
   \draw[fill, color=red] (0,2) circle (0.07)
                           (2,3) circle (0.07)
                           (4,2) circle (0.07);

   \draw[-] (0+4,0) -- (2+4,4) -- (4+4,4);            
   \draw[fill] (0+4,0) circle (0.07)
               (2+4,4) circle (0.07)
               (4+4,4) circle (0.07);

   \draw[-] (0+4,4) -- (2+4,0) -- (4+4,0);            
   \draw[fill] (0+4,4) circle (0.07)
               (2+4,0) circle (0.07)
               (4+4,0) circle (0.07);

   \draw[-,color=red] (0+4,2) -- (2+4,3) -- (4+4,2);            
   \draw[fill, color=red] (0+4,2) circle (0.07)
                           (2+4,3) circle (0.07)
                           (4+4,2) circle (0.07);
\foreach \x in {0,4,8}
 \draw[line width=1.0pt, -] (\x,-1) -- (\x,7); 

\foreach \x in {2,6}
 \draw[loosely dashed, line width=1.0pt, -] (\x,-1) -- (\x,7); 
\end{tikzpicture}
\qquad
\begin{tikzpicture}[xscale=0.4, yscale=0.25, line width = 1.5pt]
   \draw[-] (0,0) -- (2,4) -- (4,4);            
   \draw[fill] (0,0) circle (0.07)
               (2,4) circle (0.07)
               (4,4) circle (0.07);

   \draw[-] (0,4) -- (2,0) -- (4,0);            
   \draw[fill] (0,4) circle (0.07)
               (2,0) circle (0.07)
               (4,0) circle (0.07);

   \draw[-,color=red] (0,5) -- (2,3) -- (4,5);            
   \draw[fill, color=red] (0,5) circle (0.07)
                           (2,3) circle (0.07)
                           (4,5) circle (0.07);

   \draw[-] (0+4,0) -- (2+4,4) -- (4+4,4);            
   \draw[fill] (0+4,0) circle (0.07)
               (2+4,4) circle (0.07)
               (4+4,4) circle (0.07);

   \draw[-] (0+4,4) -- (2+4,0) -- (4+4,0);            
   \draw[fill] (0+4,4) circle (0.07)
               (2+4,0) circle (0.07)
               (4+4,0) circle (0.07);

   \draw[-,color=red] (0+4,5) -- (2+4,3) -- (4+4,5);            
   \draw[fill, color=red] (0+4,5) circle (0.07)
                           (2+4,3) circle (0.07)
                           (4+4,5) circle (0.07);
\foreach \x in {0,4,8}
 \draw[line width=1.0pt, -] (\x,-1) -- (\x,7); 

\foreach \x in {2,6}
 \draw[loosely dashed, line width=1.0pt, -] (\x,-1) -- (\x,7); 
\end{tikzpicture}
\caption{Positively conjugate representations.}
\label{fig:relative2}
\end{figure}
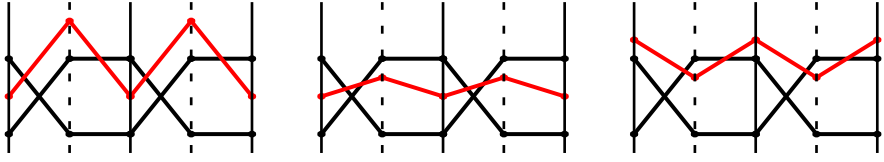

Figure \ref{fig:relative2}[left] shows shows a representation  of $(\gamma,\aset)$ and Figure \ref{fig:relative2}[right] shows a positively conjugate representation $(\gamma',\aset')$.
In the latter case $\aset' = \{3\}$, $\cset=(123)$ and $\aset' = \cset^{-1}(\{1\}) = (132)(1) = \{3\}$. It follows that $\tau_{\gamma'} = \cset\tau_{\ba,\bb} \cset^{-1} = (123)(23)(132) = (12)$. Moreover, since $\gamma' = \zeta \gamma\zeta^{-1}$, with $\zeta = (23)$ it follows that  $\tau_{\gamma'} = \zeta\tau_\gamma
\zeta^{-1} = (23)(13(23) = (12)$.
\end{exm}
\section{Discrete braid invariants}
\label{subsec:discbrinv12}
%
%
We summarize the construction of a topological invariant for 2-colored relative braid classes as described in \cite{BraidConleyIndex,GVV-pre}.
Let  $\ba\rel \bb\in \Conf_{n,m}^d$ represent  a proper, bounded discretized 2-colored braid class $[\ba\rel \bb]$. Then, the fiber $[\ba]\rel \bb$ defines
a bounded set in $\rr^{nd}$.

A sequence $\RR=\{\RR_j\}$ of functions $\RR_j\colon \rr^3\to \rr$,
 which satisfy $\partial_1\mathcal{R}_j>0$ and $\partial_3\mathcal{R}_j>0$ is called a \emph{parabolic recurrence relation}.
From   \cite[Lem.\ 55-57]{BraidConleyIndex} there exists    a parabolic recurrence relation $\RR=\{\RR_j\}$
such that $\bb$ is a zero for $\RR$, i.e. $\RR_j(x_{j-1}^{\mu_\nu},x_j^{\mu_\nu},x_{j+1}^{\mu_\nu})=0$ for all $j\in \zz$ and for all ${\nu}=1,\cdots,m$.
The recurrence relation $\RR$ may regarded as vector field and is integrated via the equations
\begin{equation}
\label{parabolicvectorfield}
 \frac{d}{ds}x^{{\mu_\nu}}_{j}=\mathcal{R}_{j}(x^{{\mu_\nu}}_{j-1},x^{{\mu_\nu}}_{j},x^{{\mu_\nu}}_{j+1}),\quad \nu=1,\cdots,m.
\end{equation}
Let $N$ denoted the closure in $\rr^{nd}$ of  $[\ba]\rel \bb$.
By   \cite[ Prop.\ 11 and Thm.\ 15]{BraidConleyIndex}, the set $N$ is an isolating neighborhood for the parabolic flow generated by
Equation \eqref{parabolicvectorfield}.
We define $\hh(\ba\rel\bb)$ as the homotopy Conley index of $\Inv(N,\RR)$, cf.\ \cite{BraidConleyIndex}, \cite{ConleyIndex}.
The Conley index is independent of the choice of parabolic recurrence relations $\RR$ for which $\RR(\bb)=0$, cf.\ \cite[Thm.\ 15(a)-(b)]{BraidConleyIndex},
as well as the choice of the fiber, i.e. $\ba\rel\bb \sim \ba'\rel\bb'$, then $\hh(\ba\rel\bb) = \hh(\ba'\rel\bb')$, cf.\ \cite[Thm.\ 15(c)]{BraidConleyIndex}.
This makes $\hh(\ba\rel\bb)$  an invariant of the discrete 2-colored braid class $[\ba\rel\bb]$.

There is an intrinsic way to define $\hh(\ba\rel\bb)$ without using parabolic recurrence relations.
We define $N^-\subset \partial N$ to be the set of boundary points for which the word
metric is locally maximal.
%
The pair $(N,N^-)$ is an index
pair for any parabolic system $\RR$ such that $\RR(\bb)=0$, and
thus by the independence of Conley index on $\RR$,  the pointed homotopy type
of
$N/N^-$ gives the Conley index: 
$
\hh(\ba\rel\bb) = [N/N^-],
$
see Figure \ref{fig:conley1} and \cite[Sect.\ 4.4]{BraidConleyIndex} for more details on the construction.

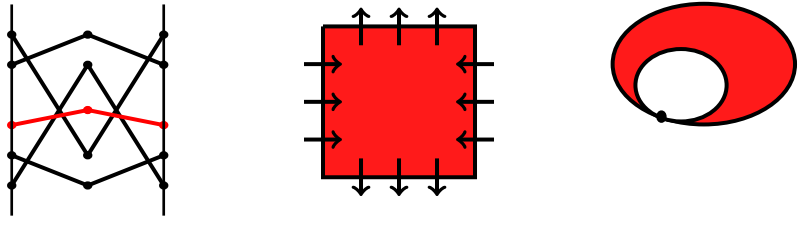
\begin{figure}[hbt]
\centering
\begin{tikzpicture}[xscale=0.5, yscale=0.4, line width = 1.5pt]
  \draw[-] (0,4) -- (2,5) -- (4,4); 
  \draw[fill] (0,4) circle (0.07)
              (2,5) circle (0.07)
              (4,4) circle (0.07);
  \draw[-] (0,1) -- (2,0) -- (4,1);
  \draw[fill] (0,1) circle (0.07)
              (2,0) circle (0.07)
              (4,1) circle (0.07);

  \draw[-] (0,5) -- (2,1) -- (4,5); 
  \draw[fill] (0,5) circle (0.07)
              (2,1) circle (0.07)
              (4,5) circle (0.07);
  \draw[-] (0,0) -- (2,4) -- (4,0);
  \draw[fill] (0,0) circle (0.07)
              (2,4) circle (0.07)
              (4,0) circle (0.07);

  \draw[-, color=red] (0,2) -- (2,2.5) -- (4,2);
  \draw[color=red,fill] (0,2) circle (0.07)
                        (2,2.5) circle (0.07)
                        (4,2) circle (0.07);
\foreach \x in {0,4}
 \draw[line width=1.0pt, -] (\x,-1) -- (\x,6); 
\end{tikzpicture}
\qquad
\qquad
\begin{tikzpicture}[xscale=0.5, yscale=0.5, line width = 1.5pt]
  \draw[-,color=white] (1,1);
  \fill[color=red!90] (2,6) -- (6,6) -- (6,2) -- (2,2) -- (2,6);
  \draw[-] (2,6) -- (6,6) -- (6,2) -- (2,2) -- (2,6); 
  \draw[->] (1.5,5) -- (2.5,5);
  \draw[->] (1.5,4) -- (2.5,4);
  \draw[->] (1.5,3) -- (2.5,3);

  \draw[<-] (1.5+4,5) -- (2.5+4,5);
  \draw[<-] (1.5+4,4) -- (2.5+4,4);
  \draw[<-] (1.5+4,3) -- (2.5+4,3);
 
  \draw[<-] (5,1.5) -- (5,2.5);
  \draw[<-] (4,1.5) -- (4,2.5);
  \draw[<-] (3,1.5) -- (3,2.5);

  \draw[->] (5,1.5+4) -- (5,2.5+4);
  \draw[->] (4,1.5+4) -- (4,2.5+4);
  \draw[->] (3,1.5+4) -- (3,2.5+4);
\end{tikzpicture}
\qquad
\qquad
\begin{tikzpicture}[xscale=0.6, yscale=0.8, line width = 1.5pt]
  \draw[-,color=white] (1,1);
  \draw[fill=red!90] (1,3.5) ellipse (2.0 and 1.0);
  \draw[fill=white] (0.5,3.15) ellipse (1.0 and 0.6);
  \draw[fill] (0.07,2.63) circle (0.07);
\end{tikzpicture}
\caption{The Conley index for the braid in Example \ref{exm:exist2}.
The homotopy of the pionted space in given by $h(\ba\rel\bb) = \sbb^1$.}
\label{fig:conley1}
\end{figure}

The invariant $\hh(\ba\rel\bb)$ is not necessarily invariant with respect to the number of discretization points $d$.
In order to have   invariance also with respect to $d$, another invariant for discrete braid classes was introduced in \cite{BraidConleyIndex}.
Consider the equivalence class induced by the relation $\ba\rel\bb \possim \ba'\rel\bb'$ on $\Conf_{n,m}^d$, which defines  the class
$[\ba\rel\bb]_{\possim}$ of proper discrete 2-colored braids.
Via the projection $\varpi\colon [\ba\rel\bb]_{\possim} \to [\bb]_{\possim}$ we obtain fibers $\varpi^{-1}(\bb)$.
Suppose $[\ba\rel\bb]_{\possim}$ is a bounded class, i.e. all fibers $\varpi^{-1}(\bb)$ are bounded sets in $\rr^{nd}$.
Following \cite[Def.\ 18]{BraidConleyIndex}
the closure $N$ of a fiber 
$\varpi^{-1}(\bb)$ is an isolating neighborhood since $\ba\rel\bb$ is proper.
Define $\HHH(\ba\rel\bb)$ as the homotopy Conley index of $N$.
If $[\ba_k]\rel \bb$ are the fibers belonging to the components $[\ba_k\rel\bb]$ of $[\ba\rel\bb]_{\possim}$, then  
\begin{equation}
\label{eqn:braidindex22}
\HHH(\ba\rel\bb) := \bigvee_{k} \hh([\ba_k]\rel\bb). 
\end{equation}
%

Define the following extension mapping
 $\E:\Conf_m^d\to\Conf_m^{d+1}$, cf.\ \cite{BraidConleyIndex},
via concatenation with the trivial braid of period one: 
\begin{equation}
	(\E\bb)^\mu := \left\{
	\begin{array}{cl}
		x_j^{\mu} &	j=0,\ldots,d; 	\\
		x_d^{\mu} &  j=d+1 .
	\end{array}\right.
\end{equation}
Properness remains unchanged under the extension mapping $\E$, however boundedness may not be preserved.
Define the skeletal augmentation:
\[
\A\colon \Conf_m^d \to \Conf_{m+2}^d,\quad \bb \mapsto \A\bb = \bb^* = \bb\cup \bb^-\cup\bb^+,
\]
where $\bb^- =\{\min_{\mu} \{x_j^\mu\} - 1\}_j$ and $\bb^+ =\{\max_{\mu} \{x_j^\mu\} + 1\}_j$.
If $[\ba\rel\bb]_{\possim}$ is bounded, then
 $\hh([\ba_k]\rel\bb) = \hh([\ba_k]\rel \bb^*)$ for all $k$ and therefore $\HHH(\ba\rel\bb) = \HHH(\ba\rel\bb^*)$.
One can define second skeletal augmentation:
\[
\B\colon \Conf_m^d \to \Conf_{m+2}^d,\quad \bb \mapsto \B\bb=\bb^\# = \bb\cup \bb^s\cup\bb^n,
\]
where $\bb^s =\{(-1)^j\min_{\mu} \{x_j^\mu\} - (-1)^j\}_j$ and $\bb^n =\{(-1)^j\max_{\mu} \{x_j^\mu\} + (-1)^j\}_j$.
As before, if $[\ba\rel\bb]_{\possim}$ is bounded, then
 $\hh([\ba_k]\rel\bb) = \hh([\ba_k]\rel \bb^\#)$ for all $k$ and therefore $\HHH(\ba\rel\bb) = \HHH(\ba\rel\bb^\#)$.

Consider the proper, bounded  2-colored braid classes $[\ba\rel\bb]_{\possim}$ and $[\E\ba\rel\E\bb]_{\possim}$.
The main result in \cite[Thm.\ 20]{BraidConleyIndex} is the Stabilization Theorem which states that
\begin{equation}
\label{eqn:braidindex13}
\HHH(\ba\rel\bb^*) = \HHH(\E\ba\rel \E\bb^*).
\end{equation}
The independence of $\HHH$ on the skeleton $\bb$ can be derived from the Stabilization Theorem.
Since a 2-colored discretized braid class is free when $d$ is sufficiently large, we have that $[\E^p\ba\rel \E^p\bb^*]$ is free
for some $p>0$ sufficiently large, and by stabilization $\HHH(\ba\rel\bb^*) = \HHH(\E^p\ba\rel \E^p\bb^*)$.
Let $\ba \rel \bb \possim \ba'\rel\bb'$, then $\E^p\ba\rel \E^p\bb^* \sim \E^p\ba'\rel \E^p\bb'^*$. 
By \cite[Thm.\ 15(c)]{BraidConleyIndex}, a continuation can be constructed which proves that
$\HHH(\E^p\ba\rel \E^p\bb^*) =\HHH(\E^p\ba'\rel \E^p\bb'^*)$. Consequently,
\[
\HHH(\ba\rel\bb^*) = \HHH(\E^p\ba\rel \E^p\bb^*) = \HHH(\E^p\ba'\rel \E^p\bb'^*) = \HHH(\ba'\rel\bb'^*),
\]
which shows that the index $\HHH$ only depends on the topological type $\llb \gamma,\aset\rrb$, with $\gamma=\beta(\ba\rel\bb)$.
\begin{defn}
\label{defn:discrbrinv}
Let $\llb\gamma,\aset\rrb$ be proper, positive conjugacy class. Then, the  \emph{braid Conley index} is defined as
\begin{equation}
\label{eqn:relbrinv}
\HHH\llb\gamma,\aset\rrb := \HHH(\ba\rel\bb^*).
\end{equation}
\end{defn}
The  braid Conley index  $\HHH$  may be computed using any
representative $\ba\rel\bb^*$ for any sufficiently large $d$ and any associated recurrence relation $\RR$.

Finally, we mention that besides the extension  $\E$, we also have a \emph{half twist} extension operator $\T$:
\begin{equation}
	(\T\bb)^\mu := \left\{
	\begin{array}{cl}
		x_j^{\mu} &	j=0,\ldots,d 	\\
		-x_d^{\mu} &  j=d+1 .
	\end{array}\right.
\end{equation}
Every discretized braid can be dualized via the mapping $\{x^\mu_j\} \mapsto \{(-1)^j x_j^\mu\}$. On $\Conf_m^{2d}$ this yields
a well-defined operator $\D\colon\Conf_m^{2d}\to \Conf_m^{2d}$  mapping proper, bounded discretized braid classes $[\ba\rel\bb]$ to proper, bounded discretized braid classes
$[\D\ba\rel\D\bb]$.
From \cite[Cor.\ 31]{BraidConleyIndex} we recall the following result.
Let $\ba\rel\bb\in \Conf_{n,m}^{2d}$ be proper, then
\begin{equation}
\label{eqn:thedual}
\HHH\bigl(\T^2\circ \D(\ba\rel\bb^*)\bigr) = \HHH\bigl(\D(\ba\rel\bb^*)\bigr) \wedge \sbb^{2n},
\end{equation}
where the wedge is the $2n$-suspension of the Conley index.

From the singular homology $H_*(\HHH(\ba\rel\bb^*))$ the Poincar\'e polynomial 
is denoted by $P_t(\ba\rel\bb^*)$, or $P_t\llb\gamma,\aset\rrb$ in terms of the topological type.
This yields an important invariant: $|P_t(\ba\rel\bb^*)| = |P_t\llb\gamma,\aset\rrb|$, which is the number of monomial term in the Poincar\'e polynomial.

\section{The variational formulation}
For a given symplectomorphism $F\in \Symp(\disc)$
 the problem of finding periodic points can be reformulated in terms of parabolic recurrence relations.


\subsection{Twist symplectomorphisms}


Let  $F(x,y) = \bigl(f(x,y),g(x,y)\bigr)$ be a symplectomorphism of $\plane$, with $f,g$ smooth functions on $\plane$.
Recall that $F\in \Symp(\plane)$ is a \emph{positive} twist symplectomorphism if
\[
\frac{\partial f(x,y)}{\partial y} >0.
\]
For twist symplectomorphisms there exists a variational principle  for finding periodic points, cf.\ \cite{LeCalvez}, \cite{Moser}.
Such a variational principle also applies to symplectomorphisms that are given as a composition: 
\[
F= F_d\circ \cdots \circ F_1,
\]
with $F_j\in \Symp(\plane)$ positive twist symplectomorphisms for all $j$.
It is important to point out that $F$ itself is \emph{not} twist in general.
An important question is whether every mapping $F\in \Symp(\plane)$ can be written  as a composition of (positive) twist symplectomorphisms, cf.\ \cite{LeCalvez}.
Suppose $F\in \Ham(\plane)$,
and $F$ allows a Hamiltonian isotopy $\psi_{t,H}$ with appropriate asymptotic conditions near infinity,
such that $\psi_{t_i,H}\circ\psi_{t_{i-1},H}^{-1}$ is close to the identity mapping in the $C^1$-norm for sufficiently small time steps $t_i-t_{i-1}$. Then, define $G_i = \psi_{t_i,H}\circ\psi_{t_{i-1},H}^{-1}$, $i=1,\cdots,k$,
and $F = G_k\circ \cdots \circ G_1$.
We remark that in this construction
 the individual mappings $G_i$ are not twist necessarily.
The following observation provides a decomposition consisting solely of positive   twist symplectomorphisms.
Consider the $90^o$ degree clockwise rotation
\[
\psi(x,y) = (y,-x), \quad \psi^4={\rm id},
\]
which is positive   twist symplectomorphism.
This yields the decomposition:
\begin{equation}
\label{eqn:decomp12}
F = (G_k\circ \psi) \circ\psi\circ\psi\circ\psi\circ \cdots \circ (G_1\circ\psi)\circ\psi\circ\psi\circ\psi,
\end{equation}
where $F_{4i} = G_{i}\circ \psi$ and $F_j=\psi$ for $j\not = 4i$ for some $i$ and $d=4k$.
Since the mappings $G_i$ are close to the identity,  the compositions $G_i\circ \psi$ are positive twist symplectomorphisms.
The above   procedure intertwines  symplectomorphisms with $k$ full rotations. As we will see later on this results in
positive braid representations of mapping classes.
The choice of $\psi$ is arbitrary since other rational rotations also yield twist symplectomorphisms.

For  symplectomorphisms $F \in \Symp(\disc)$ we establish a similar decomposition in terms of positive  twist symplectomorphisms, with the additional property that the decomposition can be extended to symplectomorphisms of $\plane$,
which is necessary  to  apply the variational techniques in \cite{BraidConleyIndex}.

\subsection{Interpolation}\label{subsec:interpl}

A symplectomorphism $F\in \Symp(\plane)$ satisfies the \emph{uniform twist condition} if the there exists a $\delta>0$ such that
\begin{equation}
\label{eqn:twist}
\delta^{-1} \ge \frac{\partial f(x,y)}{\partial y} \ge \delta >0,\quad \forall (x,y)\in \plane.
\end{equation}
The subset of such symplectomorphism is denoted by $SV(\plane)$, cf.\ \cite{LeCalvez}.
A result by Moser implies that all  symplectomorphisms of   $\plane$ with a uniform twist condition 
are Hamiltonian.
\begin{prop}[cf.\ \cite{Moser}]\label{MoserThm}
 Let $F \in SV(\plane)$. Then, there exists a Hamiltonian  $H \in \mathcal{H}(\plane)$
 such that $0<\delta \le H_{yy} \le \delta^{-1}$ and
 $\psi_{1,H} = F$, where $\psi_{t,H}$ is  the associated Hamiltonian flow.
All orbits of $\psi_{t,H}$ project to straight lines in $(t,x)$-plane, and $\psi_{t,H}\in SV(\plane)$ for all $t\in (0,1]$.
\end{prop}

For completeness we give a self-contained proof of Proposition \ref{MoserThm}, which is the same as the proof
 in \cite{Moser} modulo a few alterations.

\begin{proof}
Following \cite{Moser} we consider action
 integral $\int_{0}^{1} L(t,x(t),\dot{x}(t)) dt$
 for functions $x(t)$ with $x(0)=x_0$ and $x(1) = x_1$.
 We require that extremals are affine lines, i.e. $\ddot x(t)=0$. 
For extremals the action is given by $S(x_0,x_1) = \int_{0}^{1} L(t,x(t),\dot{x}(t)) dt$ and we seek
Lagrangians such that $S=\h$, where $\h$ is the generating function for $F$.
 For Lagrangians this implies 
%
\begin{equation}
\tfrac{d}{dt}(\partial_pL)-\partial_xL =  ( \partial_{t}+p \partial_{x} )\partial_p L-\partial_x L=0,
 \label{straightEL}
\end{equation}
where $p=\dot x$. Solving the first order partial differential equation yields
 $L=L_{0}(t,x,p)+p \partial_x m + \partial_t m$,
with
\begin{equation}
 L_{0}:=-\int_{0}^{p}(p-p')\partial^2_{x_0 x_1}\h(x-p't,x+p'(1-t)) dp' .
 \label{Fzero}
\end{equation}
and $m= m(t,x)$ to be specified later, cf.\ see \cite{Moser} for details.
%
The extremals $x(t)$ are also extremals for $L_0$. Let $S_{0}(x_{0},x_{1}) = \int_{0}^{1} L_0(t,x(t),\dot{x}(t)) dt$, then
\begin{equation}
 \int_{0}^{1}p \partial_x m(t,x(t))+\partial_t m(t,x(t)) dt = m(1,x_{1})-m(0,x_{0})
\end{equation}
and hence 
 %
\begin{equation}
 S(x_{0},x_{1})=S_{0}(x_{0},x_{1})+m(1,x_{1})-m(0,x_{0}).
\end{equation}
Differentiating $S$ yields
%
\begin{equation*}
 \partial_{x_0} S = -\partial_p L(0,x_{0},x_{1}-x_{0}),\quad
  \partial_{x_1} S = \partial_p L(0,x_{1},x_{1}-x_{0})
 \label{Sdiff}
\end{equation*}
and for the mixed derivate
\begin{equation}
 \partial^2_{x_0 x_1} S_0(x_0,x_1)=-\partial^2_{pp}L(0,x_{0},x_{1}-x_{0})=\partial^2_{x_{0} x_{1}}\h(x_{0},x_{1}).
 \label{FppRel}
\end{equation}
Then, $S_{0}(x_{0},x_{1})-h(x_{0},x_{1})=u(x_{0})+v(x_{1})$ and the choice
\begin{equation*}
 m(t,x):=(1-t)u(x)-tv(x)
\end{equation*}
implies $S=\h$. Differentiating the relation $y=-\partial_x \h(x,x_{1})$ with respect to $y$ and using the fact that
$x_1 = f(x,y)$, yields
\begin{equation*}
\begin{aligned}
1&= -\partial^2_{y x}\h (x,x_1) =  - \partial^2_{x x_{1}}\h (x,x_{1})\partial_yf(x, x_1)\\
&= -\partial^2_{x x_{1}}\h(x,x_{1})\partial_y f(x, x_{1})
 \end{aligned}
\end{equation*}
and thus $\delta\le \partial_y f \le \delta^{-1}$   if and only if $-\delta^{-1} \leq \partial^2_{x x_{1}}\h \leq -\delta$.
By relation \eqref{FppRel} we have $\partial^2_{pp}L \in [ \delta, \delta^{-1} ]$.

The Hamiltonian is obtained via the Legendre transform
\begin{equation}
 H(t,x,y):=yp-L(t,x,p),
 \label{Ltransform}
\end{equation}
where 
\begin{equation}
 y=\partial_p L(t,x,p),
 \label{Ltransform2}
\end{equation}
and we can solve for $p$, i.e. $p = \lambda(x,y)$. As before, differentiating \eqref{Ltransform2} gives
$1=\partial^2_{pp}L\cdot  \partial_y \lambda$ and differentiating \eqref{Ltransform} gives $\partial_yH = \lambda$. Combining these two identities yields $\partial^2_{pp}L\cdot \partial^2_{yy}H=1$, from which the desired property $\partial^2_{yy}H \in [ \delta, \delta^{-1} ]$ follows.

From the above analysis we obtain the following expression for the isotopy $\psi_{t,H}$:
\begin{equation}
\label{eqn:MoserIso}
\psi_{t,H}(x,y) = \Bigl(x+\lambda(x,y)t,\partial_pL\bigl(t,x+\lambda(x,y)t,\lambda(x,y)\bigr)  \Bigr).
\end{equation}
Let $\pi_x$ denote the projection onto the $x$-coordinate.~Then, $\partial_y \pi_x \psi_{t,H}(x,y) =\partial_y \lambda(x,y) t = \partial^2_{yy}H t$, which proves that
%
$\psi_{t,H}$ is positive twist for all $t\in (0,1]$.
\end{proof}

Using Proposition \ref{MoserThm} we obtain
a decomposition of symplectomorphisms $F\in \Symp(\plane)$ as given in \eqref{eqn:decomp12}
and which satisfy additional properties such that the discrete braid invariants in \cite{BraidConleyIndex} are applicable.

\begin{prop}\label{Interpolation2} 
Let  $F \in \Symp(\disc)$. Then, there exists  an  isotopy $\phi_{t} \subset \Symp(\plane)$ for all $t\in [0,1]$, an integer $d\in \nn$ and 
 a sequence $\{t_{j} \}_{j=0}^{d} \subset [0,1]$ with $t_j = j/d$, 
  such that
  \begin{enumerate}
 \item[(i)] $\phi_{0}= \id$, $\phi_{1}|_{\disc}=F$;
 \item[(ii)] $\phi_{t}$ is smooth with respect to $t$ on the intervals $[t_j,t_{j+1}]$ (piecewise smooth);
 \item[(iii)] $\widehat F_j :=\phi_{t_{j}} \circ \phi_{t_{j-1}}^{-1} \in SV(\plane)$ for all $1\le j \le d$, and $F_j:=\widehat F_j|_{\disc}$;
 \item[(iv)] the projection of the graph of $\phi_{t}(x,y)$ onto $(t,x)$-plane is linear on the intervals $t \in (t_{j-1},t_{j})$ for all $1\le j \le d$, and for all $(x,y) \in \plane$;
 \item[(v)] $\phi_{t} (\disc) \subset [-1,1] \times \rr $ for all $t\in [0,1]$;
 \item[(vi)] the points $z_\pm = (\pm 2,0)$  
 are fixed points of $\widehat F_j = \phi_{t_{j}} \circ \phi_{t_{j-1}}^{-1}$ for all $1\le j \le d$;
  \item[(vii)] the points $z'_\pm = (\pm 4,0)$  are period-2 points of $\widehat F_j = \phi_{t_{j}} \circ \phi_{t_{j-1}}^{-1}$ for all $1\le j \le d$,
  i.e. $\widehat F_j(z'_\pm) = z'_\mp = -z'_\pm$, for all $j$.
 \end{enumerate}
The decomposition
\begin{equation}
\label{eqn:decomp14}
\widehat F = \widehat F_d\circ \cdots \circ \widehat F_1,
\end{equation}
 is a generalization of the decomposition given in \eqref{eqn:decomp12}.
\end{prop}

The isotopy constructed in Proposition \ref{Interpolation2} is called a \emph{chained Moser isotopy}.
Before proving Proposition \ref{Interpolation2} we construct  analogues of the rotation mapping used in \eqref{eqn:decomp12}.

\begin{lem}\label{AlternatePsi}  
For every integer $\ell\ge 3$ there exists  
 a positive  Hamiltonian twist diffeomorphism $\Psi$ of the plane $\plane$,
such that: 
\begin{enumerate}
\item[(i)] the restriction $\Psi|_{\disc}$  is a rotation   over angle $2\pi/\ell$ and $\Psi|_{\disc}^\ell = \id$;
\item[(ii)] the points $z_\pm=(\pm 2,0)$ are fixed points for $\Psi$;
\item[(iii)] the points $z'_\pm=(\pm 4,0)$ are period-2 points for $\Psi$, i.e. $\Psi(z'_\pm) = z'_\mp$.
\end{enumerate}
\end{lem}

\begin{proof}
A linear rotation mapping on $\plane$  is a positive twist mapping for all rotation angles $\vartheta\in (0,\pi)$.
The generating function for a rotation is given by
\begin{equation}
\label{eqn:rot12}
\h_\vartheta(x,x') = \tfrac{1}{2}\cot(\vartheta) x^2 - \csc(\vartheta) x x' + \tfrac{1}{2}\cot(\vartheta) x'^2.
\end{equation}
In order to construct the mappings $\Psi$ we construct special generating functions.
Let $\ell\ge 3$ be an integer and let $\vartheta_\ell = 2\pi/\ell \in (0,\pi)$.
Consider generating functions of the form
\begin{equation}
\h_\Psi(x,x') = \xi_\ell(x) - \csc(\vartheta_\ell) x x' + \xi_\ell(x'),
\end{equation}
which generate positive twist mappings for all $\ell\ge 3$.
We choose $\xi_\ell$ as follows: $\xi_\ell(x) = \frac{1}{2}\cot(\vartheta_\ell) x^2$ for all $|x|\le 1$, $\xi_\ell(x) = \frac{1}{2}\csc(\vartheta_\ell) x^2$ for all $3/2\le |x|\le 5/2$, and  $\xi_\ell(x) = -\frac{1}{2}\csc(\vartheta_\ell) x^2$ for all $7/2\le |x|\le 9/2$.
The mapping $\Psi$ is defined by $h_\Psi$ and 
$y = -\partial_1 \h_\Psi(x,x')$ and $y' = \partial_2 \h_\Psi(x,x')$.\footnote{To simplify notation we express the derivatives of $\h$ with respect to its two  coordinates 
  by $\partial_1\h$ and $\partial_2\h$.} 
For $|x|,|x'|\le 1$, the generating function restricts to \eqref{eqn:rot12} which yields the rotation over $\vartheta_\ell$ on $\disc$ and  establishes (i).
For $3/2\le x,x'\le 5/2$ we have 
\[
y = \csc(\vartheta_\ell)(x'-x),\quad\hbox{and}\quad y' =  \csc(\vartheta_\ell)(x'-x),
\]
which verifies that $z_+$ are fixed points and same holds for $z_-$, completing the verification of (ii).
For $7/2\le x \le 9/2$ and $-9/2\le x'\le -7/2$, we have 
\[
y = \csc(\vartheta_\ell)(x'+x),\quad\hbox{and}\quad y' =  -\csc(\vartheta_\ell)(x'+x),
\]
then $z'_+$ is mapped to $z'_-$ and similarly $z'_-$ is mapped to $z'_+$, which completes (iii) and proof of the lemma.
\end{proof}

%

In order to extend chained Moser isotopies yet another type of Hamiltonian twist diffeomorphism is needed.

\begin{lem}\label{AlternatePsi3}  
For every  integer $\ell\ge 3$  there exists  
 a positive  Hamiltonian twist symplectomorphism $\Upsilon$ of the plane $\plane$,
such that: 
\begin{enumerate}
\item[(i)] the restriction $\Upsilon|_{\disc}$  is a rotation   over angle $2\pi/\ell$, i.e. $\Upsilon|_{\disc}^{\ell} = \id$;
\item[(ii)] the points $z_\pm=(\pm 2,0)$ and $z'_\pm=(\pm 4,0)$ are period-2 points for $\Upsilon$, i.e. $\Upsilon(z'_\pm) = z'_\mp$.
\end{enumerate}
\end{lem}

\begin{proof}
As before consider
 generating functions of the form
\begin{equation}
\label{eqn:rot14}
\h_\Upsilon(x,x') = \xi_\ell(x) - \csc(\vartheta_\ell) x x' + \xi_\ell(x'),
\end{equation}
which generate positive twist mappings for all   $\ell\ge 3$.
We choose $\xi_\ell$ as follows: $\xi_\ell(x) = \frac{1}{2}\cot(\vartheta_\ell) x^2$ for all $|x|\le 1$,  and  $\xi_\ell(x) = -\frac{1}{2}\csc(\vartheta_\ell) x^2$ for all $3/2\le |x|\le 9/2$.
The mapping $\Upsilon$ is defined by $h_\Upsilon$.
For $|x|,|x'|\le 1$, the generating function restricts to \eqref{eqn:rot12} which yields the rotation over $\vartheta_\ell$ on $\disc$ and  establishes (i).
For $3/2\le x,\le 9/2$ and $-9/2\le x'\le -7/2$, we have 
\[
y = \csc(\vartheta_\ell)(x'+x),\quad\hbox{and}\quad y' =  -\csc(\vartheta_\ell)(x'+x),
\]
then $z_+$ and $z'_+$ are mapped to $z_-$ and $z'_-$ respectively and similarly $z_-$ and $z'_-$ are mapped to $z_+$ and $z'_+$ respectively, which completes proof.
\end{proof}

\begin{proof}[Proof of Proposition \ref{Interpolation2}]
%
%
Consider the  subgroup $\Symp_{c}(\plane)$ formed by compactly supported symplectomorphisms of the plane.\footnote{A symplectomorphism is compactly supported in $\plane$ if it is the identity outside a compact subset of $\plane$.
}
 Recall that due to the uniform twist property the set $SV(\plane)$ is open in the topology given by $C^{1}$-convergence on compact sets, cf.\ \cite{LeCalvez}. 
 Let $\Psi\in \Ham(\plane)$ be given by Lemma \ref{AlternatePsi} for some $\ell\ge 3$.
 Then, there exists
an open neighborhood $\W \subset \Symp_{c}(\plane)$  of the identity, such that $\varphi\circ \Psi\in SV(\plane)$ for all $\varphi \in \W$.

For $F\in \Symp(\disc)$, Proposition \ref{prop:MCG11a} 
provides a Hamiltonian
 $H_{} \in \mathcal{H}(\disc)$ such that $F = \psi_{1,H}$. 
 Let $H^\dagger$ be a smooth extension to $\rr\times\plane$ and $\U_\epsilon(\disc) = \{z\in \plane~|~|z|<1+\epsilon\}$ and
let  $\alpha \colon\plane \rightarrow \rr$ 
be a smooth bump function satisfying $\alpha|_{\disc}=1$, $\alpha = 0$ on $\plane\setminus\U_\epsilon(\disc)$.
Take $\epsilon\in (0,1/2)$ and define $\widetilde H = \alpha H^\dagger$ with $\widetilde H \in \HH(\plane)$. 
The associated Hamiltonian isotopy is denoted by $\psi_{t,\widetilde H}$ and $\widetilde F = \psi_{1,\widetilde H}
\in \Ham(\plane)$. Moreover, $\psi_{t,\widetilde H}$
equals the identity on ${\plane\setminus\U_\epsilon(\disc)}$, i.e. $\psi_{t,\widetilde H}$ is supported in $\U_\epsilon(\disc)$,
and $\widetilde F|_{\disc} = F$.
%

Fix $\ell\ge 3$ and choose $k>0$
 sufficiently large  such that
the symplectomorphisms 
\begin{equation}
 G_{i}=\psi_{i/k,\widetilde H} \circ \psi_{(i-1)/k,\widetilde H}^{-1}, \ i \in \{1, \dots,k \}
\end{equation}
are elements of $\W$. Each $G_{i}$ restricted to $\disc$ can be decomposed as follows:
\begin{equation}
 G_{i}|_\disc = \bigl(G_{i}|_\disc \circ \Psi\bigr) \circ \underbrace{\Psi'\circ\cdots\circ\Psi'}_{\kappa(\ell-1)},\quad \ell\ge3,~\kappa\in \nn,
\end{equation}
where $\Psi$ and $\Psi'$ are obtained from Lemma \ref{AlternatePsi} by choosing rotation angles $2\pi/\ell$ and $2\pi/\kappa\ell$ respectively. 
Observe that $\Psi \circ \Psi'^{\kappa(\ell-1)}|_{\disc} = \id$.
%
From $\widetilde F$ we define the mapping $\widehat F\in \Symp(\plane)$:
\begin{equation} 
  \widehat{F}=(G_{k} \circ \Psi) \circ  \underbrace{\Psi' \circ \dots \circ \Psi'}_{\kappa(\ell-1)}\circ\cdots\circ (G_{1} \circ \Psi) \circ  \underbrace{\Psi'\circ\cdots\circ \Psi'}_{\kappa(\ell-1)}.
\end{equation}
By construction we have $\widehat F|_\disc = F$.
Let $\ell_\kappa= \kappa(\ell-1) +1$ and  $d = \ell_\kappa k$ and put 
\begin{equation}
  \widehat F_{j} = \begin{cases} G_{j/\ell_\kappa} \circ \Psi &\mbox{for }~~ j \in \{\ell_\kappa, 2\ell_\kappa,\cdots, d\}
    \\ \Psi' &\mbox{for } j \in \{1, \dots, d \} \backslash \{\ell_\kappa, 2\ell_\kappa, \dots, d \}. \end{cases}
\end{equation}  
with $\widehat F_j \in SV(\plane)$ for $j \in \{1, \dots, d \}$ and $F_j  = \widehat F_j|_{\disc}$.
Using the latter we obtain a decomposition of $F$ as  given in \eqref{eqn:decomp12}, 
and with the additional property that the mappings $F_j$ extend to twist symplectomorphisms of the $\plane$,
which proves \eqref{eqn:decomp14}.

Each symplectomorphism $\widehat F_j$ can be connected to identity by a Hamiltonian path.
Let $H^j$ be the Hamiltonian given by Proposition  \ref{MoserThm}, which connects
$\widehat F_j$ to the identity via the Moser isotopy $\psi_{s,H^j}$, $s\in [0,1]$.
Let $t_{j}=j/d$ for all $j \in \{0, \dots, d \}$ and define 
\begin{equation}
\label{eqn:theisotopy}
\phi_t = \psi_{s^j(t),H^j} \circ \widehat F_{j-1} \circ\cdots\circ \widehat F_0,\quad t\in [t_{j-1},t_{j}], ~~j\in \{1,\cdots,d\},
\end{equation}
with $s^j(t) = d(t-t_j)$ and $\widehat F_0 = \id$.
Observe that, by construction, $\phi_{t_j}\circ \phi_{t_{j-1}}^{-1} = \widehat F_j$, for all
$j=1,\cdots,d$ and (i) - (iv) is satisfied.
Condition (v) follows from (iv) and from the fact that each $\widehat F_j$ leaves the disc $\disc$ invariant.

   All the symplectomorphisms in the decomposition are supported in the disc $\U_\epsilon(\disc)$, hence Conditions (ii) and (iii) of  Lemma
   \ref{AlternatePsi}  
   imply Properties (vi) and (vii).
\end{proof}

\begin{rem}
\label{rem:extendedMI}
The chained Moser isotopies in Proposition \ref{MoserThm} can be extended with two more parameters  $r\ge 0$ and $\rho\ge 0$.
Consider the decoposition
\begin{equation}
\label{eqn:decomp16}
\widehat F = \widehat F_d\circ \cdots \circ \widehat F_1 \circ \underbrace{\Psi^{\ell_r}_r\circ\cdots\circ\Psi^{\ell_r}_r}_r \circ
\underbrace{\Upsilon^{\ell_\rho}_\rho\circ\cdots\circ\Upsilon^{\ell_\rho}_\rho}_\rho, 
\end{equation}
where $\Psi_r^{\ell_r}|_\disc = \id$ and $\ell_r\ge 3$, and $\Upsilon_\rho^{\ell_\rho}|_\disc = \id$ and $\ell_\rho\ge 3$.
We can again define an Moser isotopy as in \eqref{eqn:theisotopy} with $d$ replaced by $d+r\ell_r+\rho\ell_\rho$. 
The isotopy is again called a chained  Moser isotopy and denoted by $\phi_t$, and the extended period will again be denoted by $d$.
The strands $\phi_t(z_\pm)$ link with the cylinder $[0,1]\times\disc$ and with each other with  linking number $2\rho$.
\end{rem}

\subsection{The discrete action functional}\label{subsec:genfun}
%

Let $F \in \Symp(\disc)$ be the given  symplectomorphism of the 2-disc and let $\{ \phi_{t} \}_{t \in \rr}$ and $\{t_{j} \}_{j \in \{ 0, \dots, d \} }$ 
be the associated continuous isotopy and   sequence of discretization times as given in Proposition~\ref{Interpolation2}
for the extension $\widehat F$.
The isotopy is extended periodically, that is $ \phi_{t+s} = \phi_{t} \circ \phi_{1}^{s} $ and $t_{j+s d} = s + t_{j}$ for all $s \in \mathbb{Z}$.
The decomposition of $\widehat F$ given by Proposition \ref{Interpolation2} yields a periodic sequence of positive twist symplectomorphisms $\{\widehat F_j\}$, with $\widehat F_j =
\phi_{t_{j+1}} \circ \phi_{t_{j}}^{-1} \in SV(\plane)$ and $\widehat F_{j+d} = \widehat F_j$.


\begin{defn}
A sequence $\{ (x_{j},y_{j}) \}_{j \in \mathbb{Z}}$ is a \textit{full orbit} for the system $\{ \widehat F_j\}$ if
\[
(x_{j+1},y_{j+1}) =   \widehat F_j (x_{j},y_{j}), \quad j\in \zz.
\]
If $(x_{j+n},y_{j+d}) = (x_j,y_j)$ for all $j$, then $\{ (x_{j},y_{j}) \}_{j \in \mathbb{Z}}$ is called an \emph{$d$-periodic sequence}
for the system $\{ \widehat F_j\}$.
\end{defn}

%

  For every twist symplectomorphism $\widehat F_j  \in SV(\plane)$ 
we assign a generating function
  $\h_{j}=\h_{j}(x_{j}, x_{j+1})$ on the $x$-coordinates, which implies that
  $y_{j}= -\partial_1 \h_j$ 
  and $y_{j+1}=\partial_2 \h_j$. 
  From the twist property it follows that
  \begin{equation}\label{hmonotonicity}
    \partial_{1} \partial_{2} \h_{j} < 0,\quad \forall j\in \zz.
  \end{equation}
  Note that the sequence $\{\h_{j} \}$ is $d$-periodic.  
 
  Define the \textit{action functional} $W_{d}:\rr^{\mathbb{Z}/d\mathbb{Z}} \rightarrow \rr$ by
  \begin{equation}
  \label{eqn:action1}
    W_{d}\bigl(\{x_j\}\bigr):= \sum\limits_{j=0}^{d-1} \h_{j} (x_j, x_{j+1}).
  \end{equation}
A sequence $\{x_j\}$ is a critical point of $W_d$ if and only if
\begin{equation}
\label{eqn:parabrec}
 \mathcal{R}_{j}(x_{j-1}, x_{j}, x_{j+1}) :=    -\partial_{2} \h_{j-1} \left(x_{j-1}, x_{j}\right) - \partial_{1} \h_j \left(x_{j}, x_{j+1}\right)=0,
\end{equation}
    for all $j \in \mathbb{Z}$.
    The $y$-coordinates satisfy $y_{j}=\partial_{2} \h_{j-1}\left(x_{j-1},x_{j}\right)$. 
 
  Periodicity and exactness of $\mathcal{R}_j$ is immediate. The monotonicity follows directly from inequality~\eqref{hmonotonicity}.
A periodic point $z$, i.e. $F^d(z) = z$, is equivalent to the periodic sequence $\{ (x_{j},y_{j}) \}_{j \in \mathbb{Z}}$, with $z=(x_0,y_0)\in \disc$.
Since $z=(x_0,y_0)\in \disc$, the invariance of $\disc$ under $F$ implies that $(x_{j},y_{j})\in \disc$ for all $j$.
The above considerations yield the following variational principle.
\begin{prop}
\label{prop:varprin1}
A $d$-periodic sequence $\{ (x_{j},y_{j}) \}_{j \in \mathbb{Z}}$ is an $d$-periodic orbit for the system $\{ \widehat F_j\}$ if and only if
the sequence of $x$-coordinates $\{ x_{j} \}$ is a critical point of $W_d$. 
\end{prop}

The idea of periodic sequences can be generalized to periodic configurations.
Let $\{B_j\}_{j\in \zz}$, $B_j = \{(x_j^\mu,y_j^\mu)~|~\mu=1,\cdots,m\}\in \C_m(\plane)$ and $B_{j+d} = B_j$ for all $j$. 
Such a sequence $\{B_j\}$ is a $d$-periodic sequence for $\{\widehat F_j\}$ if $\widehat F_j(B_j) = B_{j+1}$ for all $j\in \zz$.

For a $d$-periodic sequence $\{B_j\}$, the $x$-projection yields a discretized braid $\bb = \{\bb^\mu\} = \{x_j^\mu\}$, cf.\ Definition \ref{PL}.
The above action functional can be extended to the space of discretized braids $\Conf_m^d$:
\begin{equation}
\label{eqn:action2}
W_d(\bb) := \sum_{\mu=1}^m W_d(\bb^\mu), 
\end{equation}
where $W_d(\bb^\mu)$ is given by \eqref{eqn:action1}. This yields the following extension of the variational principle.
\begin{prop}
\label{prop:varprin2}
A $d$-periodic sequence $\{ B_j \}_{j \in \mathbb{Z}}$, $B_j\in \C_m(\plane)$, is a $d$-periodic sequence of configurations for the system $\{ \widehat F_j\}$ if and only if
the sequence of $x$-coordinates $\bb=\{ x^\mu_{j} \}$ is a critical point of $W_d$ on $\Conf_m^d$. 
\end{prop}

A discretized braid $\bb$ that is stationary for $W_d$ if it satisfies the parabolic recurrence relations in \eqref{eqn:parabrec} for all $\mu$
and   the periodicity condition in Definition \ref{PL}(b).
In Section \ref{sec:braiding} we show that $d$-periodic sequences of configurations $\{B_j\}$ for the system $\{ \widehat F_j\}$
yields  geometric braids. 

\section{Braiding of periodic points}
\label{sec:braiding}
For  symplectomorphisms $F\in \Symp(\disc)$, with a finite invariant set $B\subset \inter \disc$, the mapping class
can be identified via a chained Moser isotopy.

\begin{prop}
\label{prop:traceout1}
Let $B\subset\inter\disc$, with $\# B=m$,  be a finite invariant set for $F\in \Symp(\disc)$ and let
$\phi_t$ be a chained Moser isotopy given in Proposition \ref{Interpolation2}. Then, $\bbeta(t) = \phi_t(B)$ represents a   geometric  braid based at $B\in \C_m\disc$     with only positive crossings and $\beta = \imath_B\bigl([\bbeta]_B\bigr)$ is a
positive word  in
the braid monoid $\BB_m^+$.
The $x$-projection $\bb(t) = \pi_x\bbeta(t)$ on the $(t,x)$-plane is a (continuous)   piecewise linear braid diagram.
\end{prop}

 Proposition \ref{prop:MCG12} implies that the associated positive braid word $\beta\in \mathcal{B}_m$, derived from the braid diagram $\pi_x\phi_t(B)$ determines the mapping class of $F$
relative to $B$.
If the based path $\bbeta(t)=\phi_t(B)$ is regarded as a \emph{free loop}  $\sbb^1 \to \C_m\disc$, i.e. discarding the base point, then $\bbeta$ is referred to as a \emph{closed} geometric braid in $\disc$.
\begin{defn}
\label{defn:acylindrical}
Let $\bbeta$ be a geometric braid in $\disc$. A component of $\bbeta'\subset \bbeta$ is called \emph{cylindrical in} $\bbeta$ if $\bbeta'$ can be deformed onto $\partial \disc$ as a closed geometric braid. Otherwise $\bbeta'$ is called \emph{acylindrical}. A union of components $\bbeta'$ is called cylindrical/acylindrical in $\bbeta$ if all members are.
\end{defn}

\begin{rem}
\label{rmk:acylindrical}
A  positive conjugacy class $\llb \gamma,\aset\rrb$ is associated with braid classes in $[\ba\rel\bb]_{\possim}$ 
in $\Conf_{n,m}^{d+q}$, cf.\ Section \ref{subsec:algpres}. If for a representative $\ba\rel\bb$ it holds that $\Bd(\ba)$ is cylindrical/acylindrical
in $\bgamma =\Bd(\ba)\rel\Bd(\bb)$,
then $\llb \gamma,\aset\rrb$ is said to  cylindrical/acylindrical, cf.\  Definition \ref{defn:proper2}.
\end{rem}


Let $z,z'\in \plane$ be distinct points with the property that $\widehat F^n(z) = z$ and $\widehat F^n(z') = z'$, for some $n\ge 1$, and where $\widehat F = \phi_1$ and
$\phi_t$ a chained Moser isotopy constructed in Proposition \ref{Interpolation2}.
Define the continuous functions $z(t) = \phi_t(z)$ and $z'(t) = \phi_t(z')$ and let $x(t)$ and $x'(t)$ the $x$-projection of $z(t)$ and $z'(t)$ respectively.
By Proposition \ref{Interpolation2}, $x(t)$ and $x'(t)$ are (continuous) piecewise linear functions that are uniquely determined by the sequence
$\{t_j\}_{j=0}^{nd}$, $t_j = j/d$.

\begin{lem}[cf.\ \cite{BraidConleyIndex}]
\label{lem:propergraphs}
The two $x$-projections $x(t)$ and $x'(t)$ form a (piecewise linear) braid diagram, i.e.\ no tangencies.
The intersection number $\iota\bigl(x(t),x'(t)\bigr)$, given as the total number of intersections of the graphs of $x(t)$ and $x'(t)$
on the interval $t\in [0,n]$,
 is well-defined and even.
\end{lem}

\begin{proof}
Let $x_j = x(t_j)$ and $x_j' = x'(t_j)$, $j=0,\cdots,nd$ and by the theory in  Section \ref{subsec:genfun} the sequences satisfy the parabolic recurrence relations
$\mathcal{R}_{j}(x_{j-1}, x_{j}, x_{j+1}) =0$ and $\mathcal{R}_{j}(x'_{j-1}, x'_{j}, x'_{j+1}) =0$.
Suppose the sequences $\{x_j\}$ and $\{x'_j\}$ have a tangency at $x_j = x_j'$ (but are not identically equal). Then, either $x'_{j-1}<x_{j-1}$ and $x'_{j+1} \le x_{j+1}$, or
$x'_{j-1}\le x_{j-1}$ and $x'_{j+1} < x_{j+1}$, and similar with the role of $\{x_j\}$ and $\{x'_j\}$ reversed.
Since functions $\mathcal {R}_j$ are strictly increasing in the first and third variables and since $x_j = x'_j$ both evaluations of $\mathcal{R}_j$ cannot be zero simultaneously,
which contradicts the existence of tangencies. All intersections of $x(t)$ and $x'(t)$ are therefore transverse in the sense of Definition \ref{PL}(c) and thus 
$\iota\bigl(x(t),x'(t)\bigr)$ is well-defined and even.
\end{proof}

The curves $z(t)$ and $z'(t)$ may be regarded as 3-dimensional (continuous) curves $z,z'\colon \rr/\zz \to \rr^3$, $t\mapsto (t,z(t))$ and $t\mapsto (t,z'(t))$.
Due to the special properties of the chained Moser isotopy $\phi_t$ we have:
\begin{lem}
\label{lem:linking1}
The graphs $t\mapsto (t,z(t))$ and $t\mapsto (t,z'(t))$ form a positive 2-strand braid.  Intersections in the $x,x'$-braid diagram correspond to positive crossings in the $z,z'$-braid. The linking number is given by
${\rm Link}\bigl(z(t),z'(t)\bigr) =\frac{1}{2} \iota\bigl(x(t),x'(t)\bigr)$.
\end{lem}

\begin{proof}
By Lemma \ref{lem:propergraphs} the projection graphs $t\mapsto x(t)$ and $t\mapsto x'(t)$ form a braid diagram.
In order to show that the graphs $t\mapsto (t,z(t))$ and $t\mapsto (t,z'(t))$ form a positive 2-strand braid we examine the intersections in
the $x,x'$-braid diagram.

(i) Consider an intersection on the interval $[t_j,t_{j+1}]\subset [0,n]$ 
for which $x_j<x'_j$ and $x_{j+1}>x'_{j+1}$.
Let $\tau\in [t_j,t_{j+1}]$ be the intersection point and $x(\tau) = x'(\tau) = x_*$.
After rescaling and shifting to the interval $[0,1]$ we have
$x(s(t)) = x_j + (x_{j+1}-x_j)s(t)$, $s(t) = d(t-t_j) \in [0,1]$ and the same for $x'(s(t))$.
Recall that $\phi_t$ is given by  
\eqref{eqn:theisotopy} and therefore by \eqref{Ltransform2}, 
\[
y(s(\tau)) = \partial_p L^j\bigl(s(\tau),x_*,x_{j+1}-x_j), \quad y'(s(\tau)) = \partial_pL^j\bigl(s(\tau),x_*,x'_{j+1}-x'_j),
\]
where $L^j$ are the Lagrangians for the Moser isotopies $\psi_{t,H^j}$ in Proposition \ref{MoserThm}.
Since $\partial^2_{pp}L^j\ge \delta>0$ and $x_{j+1}-x_j > x'_{j+1}-x'_j$, we conclude that $y(s(\tau)) > y'(s(\tau))$.
By reversing the role of $x$ and $x'$, i.e.\ $x_j>x'_j$ and $x_{j+1}<x'_{j+1}$, we obtain $y(s(\tau))<y'(s(\tau))$,
which shows that an intersection in the $x,x'$-diagram
corresponds to a positive crossing in the $z,z'$-braid.

(ii) Consider an intersection at $x_j$, for which   $x_{j-1}<x'_{j-1}$, $x_j = x'_j = x_*$ and $x_{j+1}>x'_{j+1}$.
As in the previous case
\[
\begin{aligned}
y(s(\tau)) &= \partial_pL^{j-1} \bigl(1,x_*,x_*-x_{j-1}) = \partial_pL^j\bigl(0,x_*,x_{j+1}-x_*);\\
y'(s(\tau)) &= \partial_pL^{j-1} \bigl(1,x_*,x_*-x'_{j-1}) = \partial_pL^j\bigl(0,x_*,x'_{j+1}-x_*),
\end{aligned}
\]
and since $x_*-x_{j-1}>x_*-x'_{j-1}$ (and $x_{j+1}-x_* > x'_{j+1}-x_*$) we conclude that $y(s(\tau)) >y'(s(\tau))$. Reversing the role of $x$ and $x'$ yields
$y(s(\tau))<y'(s(\tau))$. In this case we also conclude that an intersection in the $x,x'$-diagram
corresponds to a positive crossing in the $z,z'$-braid, which concludes the proof.
\end{proof}

\begin{proof}[Proof of Proposition \ref{prop:traceout1}]
Since $\phi_t$ is an isotopy and $F(B) = \phi_1(B) = B$, the path $t\mapsto \phi_t(B)$ in $\C_m\disc$ represents a geometric braid $\bbeta(t)$.
By Lemma \ref{lem:linking1} all crossings in $\bbeta(t)$ are positive. Indeed, if we consider the $m$-fold cover $\widehat \bbeta(t)$ of $\bbeta$, i.e.\
$t\mapsto \phi_t(B)$, $t\in [0,m]$, then all pairs of strands satisfy the hypotheses of Lemma \ref{lem:linking1}.
Consequently, the presentation of $\beta\in \BB_m$ of $\bbeta$ is unique and consists of only positive letters.
\end{proof}

Our decomposition automatically selects a braid word in $\BB_m^+$ which is positive and which may be represented as a positive
piecewise linear braid diagram.
This allows us to use the theory of parabolic recurrence relations for finding additional periodic points for $F$. 
In the following lemma $\widehat F$ is an extension of $F$ to $\plane$ given by Proposition \ref{Interpolation2}.

\begin{lem}
\label{prop:reallyindisc}
Let $A,B\subset \plane$ be finite, disjoint sets  and   $B \subset\inter\disc$. 
Let $F\in \Symp(\disc)$ with $F(B) = B$ and let 
  $\phi_t$ be a chained Moser isotopy given by Proposition \ref{Interpolation2} with $\phi_1 = \widehat F$. Suppose  
$A$ is an invariant set for $\widehat F$ and $\balpha(t) = \phi_t(A)$ is acylindrical in  $\balpha\rel\bbeta$.
%
 Then, $A$ is an invariant set for $F$ with $A\subset \inter\disc$.
\end{lem}

\begin{proof}
The set $A = \balpha(0)$ is an invariant set for $\widehat F$. Assume without loss of generality that $\balpha(t)$ is a single component braid.
Let $\balpha^\dagger(t)$ be the $n$-fold cover of $\balpha$, with $t\in [0,n]$. If $\balpha^\dagger(t_j) \in \plane \setminus \inter\disc$, $t_j=j/d$, for some $j$, then
$\balpha^\dagger(t_j) \in \plane \setminus \inter\disc$ for all $j$, since the set $\phi_t(\disc)$ separates the points inside and outside $\partial\disc$ under the isotopy $\phi_t$. Therefore, $\balpha^\dagger(t) \in \plane\setminus\inter \phi_t(\disc)$
for all $t\in [0,n]$ and
thus $\balpha(t) \in  \plane\setminus \inter \phi_t(\disc)$ for all $t\in [0,1]$.
By assumption $\bbeta(t) \in   \phi_t(\disc)$ for all $t\in [0,1]$
and therefore  that $\balpha$ can be contracted onto $\partial \phi_t(\disc)$, 
which contradicts the assumption that $\balpha \rel\bbeta$ is acylindrical.
Consequently, $\balpha(t) \in   \phi_t(\disc)$ for all $t\in [0,1]$.
The latter implies $A\subset \inter\disc$, which completes the proof.
\end{proof}


\section{The main theorem}
\label{sec:main1}



Given $F\in \Symp(\disc)$ and  a finite invariant set $B\subset \inter \disc$,
then 
\begin{equation}
\label{eqn:detMP}
\jmath_B([F]) = \beta\!\! \!\!\mod\square \in \BB_m/Z(\BB_m).
\end{equation}
%
The braid word $\beta$ can chosen positive by adding full twists and can be used to force additional invariant sets as quantified by the  following result.

\begin{thm}
\label{thm:main1}
Let $B\subset \inter\disc$ be a finite invariant set for a symplectomorphism $F\colon \disc\to \disc$ 
and  let $\beta\in
\BB_m^+$  be  a positive braid word representing the mapping class of $F$ in $\Mod(\disc\rel B)$.
Consider a coloring $\aset\subset \{1,\cdots,n+m\}$ of length $n$ and a 2-colored braid $(\gamma,\aset) \in \pi_\aset^{-1}(\beta) \subset \BB^+_{n,m,\aset}$
and assume that $\llb\gamma,\aset\rrb$ is proper and acylindrical,\footnote{See Definition \ref{defn:proper2} and Remark \ref{rmk:acylindrical}.}
and $\HHH\llb\gamma,\aset\rrb\not = 0$.
%
%
Then, 
\begin{enumerate}
\item [(i)] there exists a finite invariant set $A\subset \inter\disc$ of $F$
such that the mapping class of $F$ satisfies 
$\jmath_{A\cup B}([F]) = \gamma\!\!\mod \square$; 
\item [(ii)] the number of distinct invariants sets $A$ in (i) is bounded from below by $|P_t\llb\gamma,\aset\rrb|$.
\end{enumerate}
\end{thm}

\begin{proof}
Recall from Section \ref{subsec:algpres} that 
$\llb\gamma,\aset\rrb$ determines a 2-colored discretized braid class $[\ba\rel\bb]_{\possim}$ via
the mapping $\gamma \mapsto \ev_0(\gamma) \in \Conf_{n+m}^{d_0}$, with $d_0$ the number of generators in $\gamma$. 
The representative $\ba \rel \bb\in \Conf_{n,m}^{d_0}$ is obtained from the coloring $\aset$.
%
%
%

Consider a chained Moser isotopy $\phi_t$ with $r=\rho=0$, then 
\[
\imath_B([\phi_t(B)])\!\!\!\! \mod \square= \beta,
\]
where $\ell\ge 3$ and $k$ in Proposition \ref{Interpolation2} are fixed. 
Let $\bb_B := \bigl\{ \pi_x \phi_{t_j}(B)\bigr\}$, $j=0,\cdots,d$ and 
choose $\kappa$ large enough such that $[\bb_B]$ defines a free discrete braid class. 
The following
 three cases can be distinguished:
\[ 
\bb \possim \bb_B,\quad \bb\possim \T^{2\lambda} \bb_B,\quad\hbox{and}\quad \T^{2\lambda} \bb \possim \bb_B,
\] 
for some $\lambda\ge 1$. The choices of $\ell$, $k$ and $\kappa$ determine $d$.

 {\it Case I:} 
 The integers $q\ge 0$ and  $\kappa$ are chosen large enough such that $\E^q\bb^*\sim \bb_B^*$ and $[\E^q(\ba\rel\bb^*)]$ is free,
 where $\bb_B^* = \bb_B \cup \bb_B^-\cup \bb_B^+$, with $\bb_B^\pm = \{\pm 2\}$. Note that $\bb_B^* = \bigl\{ \pi_x \phi_{t_j}(B^*)\bigr\}$,
 with $B^* = B\cup \{-2,+2\}$.
 Choose $\ba_B\rel\bb_B^* \sim \E^q(\ba\rel\bb^*)$, cf.\ Figure \ref{fig:7C1},  then by the invariance of the discrete braid invariant we have that
\[
\HHH(\ba_B\rel\bb_B^*) = \HHH\bigl(\E^q(\ba\rel\bb^*)\bigr) = \HHH(\ba\rel\bb^*) \not = 0,
\]
which proves that $P_t(\ba_B\rel\bb_B^*) = P_t(\ba\rel\bb^*) =  P_t\llb\gamma,\aset\rrb\not = 0$.

From the Morse Theory in \cite[Lemma 35]{BraidConleyIndex} we derive that the number of stationary 
discrete braid diagrams $\ba_B\rel\bb_B$ for the action $W_d$ in \eqref{eqn:action1} induced by the Moser isotopy $\phi_t$ 
is bounded below by $|P_t\llb\gamma,\aset\rrb|$.
The associated piecewise linear braid diagrams $\Bd(\ba_B\rel\bb_B)$ 
lift to 2-colored braids $\balpha_B(t)\rel\bbeta_B(t)$ by
constructing the $y$-component via \eqref{eqn:MoserIso}.
The stationary braids $\balpha_B(t)\rel\bbeta_B(t)$ are in fact braids in $\disc$ by Proposition \ref{prop:reallyindisc} and
$A=\balpha_B(0) \subset \inter\disc$ is an invariant set for $F$. Moreover, $\balpha_B(t)\rel\bbeta_B(t)$
determines the mapping class of $F$, which completes the proof in Case I.

\begin{figure}
\centering
\begin{tikzpicture}[xscale=0.23, yscale=0.23, line width = 1.5pt]
\foreach \x in {-2,8,16}
 \draw[line width=1.0pt, -] (\x,-3) -- (\x,7); 

 \draw[line width=1.0pt,decorate,decoration={brace,amplitude=8pt}] (16.0,-3.5) -- (8.1,-3.5) node [black,midway,yshift=-0.5cm] {\footnotesize $q$};
 \draw[line width=1.0pt,decorate,decoration={brace,amplitude=8pt}] (7.9,-3.5) -- (-2.,-3.5) node [black,midway,yshift=-0.5cm] {\footnotesize $d_0$};
 \draw[line width=1.0pt,decorate,decoration={brace,amplitude=8pt}] (-2,7.5) -- (16.,7.5) node [black,midway,yshift=0.5cm] {\footnotesize $\E^q(\ba\rel\bb^*)$};
   
   \draw[-] (-2,2) -- (0,0) -- (2,0) -- (4,2) -- (6,4) -- (8,4) -- (10,4) -- (12,4) -- (14,4) -- (16,4);            
   \draw[fill] (-2,2) circle (0.07)
               (0,0) circle (0.07)
               (2,0) circle (0.07)
               (4,2) circle (0.07)
               (6,4) circle (0.07);

\foreach \x in {8,10,12,14,16}
   \draw[fill] (\x,4) circle (0.07);

   \draw[-] (-2,4) -- (0,4) -- (2,2) -- (4,0) -- (6,0) -- (8,2) -- (10,2) -- (12,2) -- (14,2) -- (16,2);            
   \draw[fill] (0,4) circle (0.07)
               (2,2) circle (0.07)
               (4,0) circle (0.07)
               (6,0) circle (0.07)
               (-2,4) circle (0.07);

\foreach \x in {8,10,12,14,16}
   \draw[fill] (\x,2) circle (0.07);

   \draw[-,color=red] (-2,0) -- (0,2) -- (2,4) -- (4,4) -- (6,2) -- (8,0) -- (10,0) -- (12,0) -- (14,0) -- (16,0);            
   \draw[fill, color=red]  (0,2) circle (0.07)
                           (2,4) circle (0.07)
                           (4,4) circle (0.07)
                           (6,2) circle (0.07)
                           (-2,0) circle (0.07);

\foreach \x in {8,10,12,14,16}
   \draw[fill, color=red] (\x,0) circle (0.07);

\foreach \y in {-2,6}
 \draw[-] (-2,\y) -- (16,\y);

\foreach \x in {-1,...,8}
   \draw[fill] (2*\x,-2) circle (0.07);

\foreach \x in {-1,...,8}
   \draw[fill] (2*\x,6) circle (0.07);

\end{tikzpicture}
\raisebox{55pt}{
\begin{minipage}{0.05\linewidth}
  \centering \Large $\sim$
\end{minipage}}
\begin{tikzpicture}[xscale=0.23, yscale=0.23, line width = 1.5pt]

\draw[line width=1.0pt,decorate,decoration={brace,amplitude=8pt}] (16,-3.5) -- (-2.,-3.5) node [black,midway,yshift=-0.5cm] {\footnotesize $d$};

\draw[line width=1.0pt,decorate,decoration={brace,amplitude=8pt}] (-2,7.5) -- (16.,7.5) node [black,midway,yshift=0.5cm] {\footnotesize $\ba_B\rel\bb_B^*$};

\foreach \x in {-2,16}
 \draw[line width=1.0pt, -] (\x,-3) -- (\x,7); 

\draw[-] (-2,2) -- (0,2.5) -- (2,2.9) -- (4,2.6) -- (6,2.7) -- (8,4) -- (10,4.7) -- (12,4.5) -- (14,3.8) -- (16,4);            
\foreach \point in {(-2,2), (0,2.5), (2,2.9), (4,2.6), (6,2.7), (8,4), (10,4.7), (12,4.5), (14,3.8), (16,4)}
   \draw[fill] \point circle (0.07);
  
\draw[-] (-2,4) -- (0,4.2) -- (2,4) -- (4,3.7) -- (6,3.5) -- (8,3.4) -- (10,3) -- (12,2.8) -- (14,2.2) -- (16,2);            
\foreach \point in {(-2,4), (0,4.2), (2,4), (4,3.7), (6,3.5), (8,3.4), (10,3), (12,2.8), (14,2.2), (16,2)}
   \draw[fill] \point circle (0.07);

\draw[-,color=red] (-2,0) -- (0,0.5) -- (2,2) -- (4,4) -- (6,5) -- (8,5) -- (10,4) -- (12,2) -- (14,1) -- (16,0);            
\foreach \point in {(-2,0), (0,0.5), (2,2), (4,4), (6,5), (8,5), (10,4), (12,2), (14,1), (16,0)}
   \draw[fill, color=red] \point circle (0.07);

\foreach \y in {-2,6}
 \draw[-] (-2,\y) -- (16,\y);

\foreach \x in {-1,...,8}
   \draw[fill] (2*\x,-2) circle (0.07);

\foreach \x in {-1,...,8}
   \draw[fill] (2*\x,6) circle (0.07);

\end{tikzpicture}
\caption{Two representatives for Case I. Numbers of discretization points are linked by $d=d_0+q$.}
\label{fig:7C1}
\end{figure}
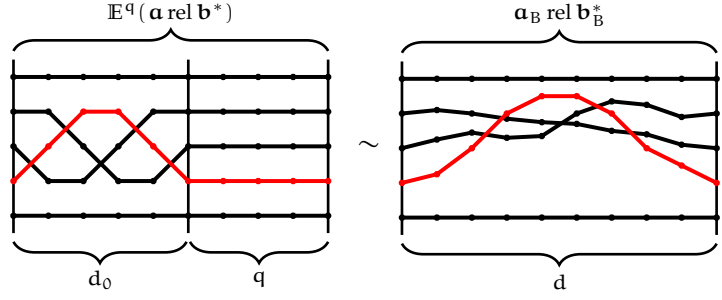


 {\it Case II:}
 By Remark \ref{rem:extendedMI} we
 use a(n) (extended) chained Moser isotopy,  denoted by $\widetilde\phi_t$, with $r=\lambda$, $\ell_r\ge 3$ and $\rho=0$
 and we denote the associated discrete braid by $\widetilde \bb_B:= \bigl\{ \pi_x \widetilde\phi_{t_j}(B)\bigr\}$.
 By construction $\widetilde \bb_B \possim \T^{2\lambda}\bb_B \possim \bb$.
  The integers $q\ge 0$ and  $\kappa$ are chosen large enough such that $\E^{q+2\lambda}\bb^* \sim \widetilde \bb_B^*$ and $[\E^{2\lambda+q}(\ba\rel\bb^*)]$ is free,  where $\widetilde \bb_B^* = \widetilde \bb_B \cup \bb_B^-\cup \bb_B^+ = \bigl\{ \pi_x \widetilde \phi_{t_j}(B^*)\bigr\}$.
As in Case I we choose $\ba_B\rel \widetilde\bb_B^*\sim \E^{2\lambda+q}(\ba\rel\bb^*)$, cf.\ Figure \ref{fig:7C2}, such that 
\[
\HHH\bigl(\ba_B\rel \widetilde \bb_B^*\bigr) = \HHH\bigl( \E^{2\lambda+q}(\ba\rel\bb^*)\bigr) = \HHH(\ba\rel\bb^*) \not = 0.
\]
This proves the theorem in Case II.
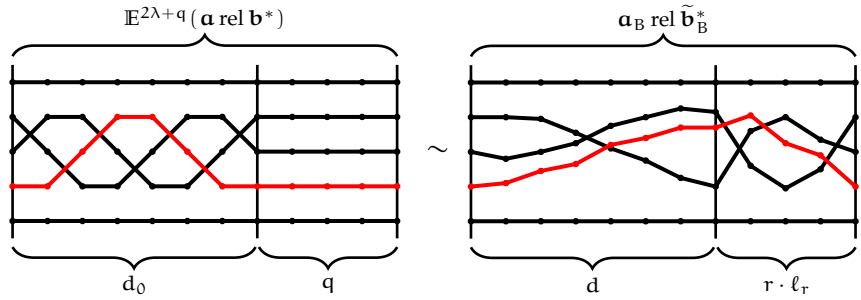
\begin{figure}
  \centering
\begin{tikzpicture}[xscale=0.23, yscale=0.23, line width = 1.5pt]

\foreach \x in {-4,10,18}
 \draw[line width=1.0pt, -] (\x,-3) -- (\x,7); 

\draw[-] (-4,4) -- (-2,2) -- (0,0) -- (2,0) -- (4,2) -- (6,4) -- (8,4) -- (10,2) -- (12,2) -- (14,2) -- (16,2) -- (18,2);            
\foreach \point in {(-4,4), (-2,2), (0,0), (2,0), (4,2), (6,4), (8,4), (10,2), (12,2), (14,2), (16,2), (18,2)}
   \draw[fill] \point circle (0.07);

\draw[-] (-4, 2) -- (-2,4) -- (0,4) -- (2,2) -- (4,0) -- (6,0) -- (8,2) -- (10,4) -- (12,4) -- (14,4) -- (16,4) -- (18,4);            
\foreach \point in {(-4, 2), (-2,4), (0,4), (2,2), (4,0), (6,0), (8,2), (10,4), (12,4), (14,4), (16,4), (18,4)}
   \draw[fill] \point circle (0.07);

\draw[-,color=red] (-4,0) -- (-2,0) -- (0,2) -- (2,4) -- (4,4) -- (6,2) -- (8,0) -- (10,0) -- (12,0) -- (14,0) -- (16,0) -- (18,0);            
\foreach \point in {(-4,0), (-2,0), (0,2), (2,4), (4,4), (6,2), (8,0), (10,0), (12,0), (14,0), (16,0), (18,0)}
   \draw[fill, color=red] \point circle (0.07);

 \draw[line width=1.0pt,decorate,decoration={brace,amplitude=8pt}] (18.0,-3.5) -- (10.1,-3.5) node [black,midway,yshift=-0.5cm] {\footnotesize $q$};
 \draw[line width=1.0pt,decorate,decoration={brace,amplitude=8pt}] (9.9,-3.5) -- (-4.,-3.5) node [black,midway,yshift=-0.5cm] {\footnotesize $d_0$};
 \draw[line width=1.0pt,decorate,decoration={brace,amplitude=8pt}] (-4,7.5) -- (18.,7.5) node [black,midway,yshift=0.5cm] {\footnotesize $\E^{2\lambda+q}(\ba\rel\bb^*)$};

\foreach \y in {-2,6}
 \draw[-] (-4,\y) -- (18,\y);

\foreach \x in {-2,...,9}
   \draw[fill] (2*\x,-2) circle (0.07);

\foreach \x in {-2,...,9}
   \draw[fill] (2*\x,6) circle (0.07);
\end{tikzpicture}
\raisebox{55pt}{
\begin{minipage}{0.05\linewidth}
  \centering \Large $\sim$
\end{minipage}}
\begin{tikzpicture}[xscale=0.23, yscale=0.23, line width = 1.5pt]

\foreach \x in {-4,10,18}
 \draw[line width=1.0pt, -] (\x,-3) -- (\x,7); 

\draw[-] (-4,2) -- (-2,1.6) -- (0,2.) -- (2,2.5) -- (4,3.5) -- (6,4.) -- (8,4.5) -- (10,4.3) -- (12,1.2) -- (14,-0.1) -- (16,1.0) -- (18,4);            
\foreach \point in {(-4,2), (-2,1.6), (0,2.), (2,2.5), (4,3.5), (6,4.), (8,4.5), (10,4.3), (12,1.2), (14,-0.1), (16,1.0), (18,4)}
   \draw[fill] \point circle (0.07);
  
\draw[-] (-4,4) -- (-2,4) -- (0,3.9) -- (2,3.1) -- (4,2.2) -- (6,1.5) -- (8,0.5) -- (10,0) -- (12,3.2) -- (14,4.0) -- (16,2.7) -- (18,2);            
\foreach \point in {(-4,4), (-2,4), (0,3.9), (2,3.1), (4,2.2), (6,1.5), (8,0.5), (10,0), (12,3.2), (14,4.0), (16,2.7), (18,2)}
   \draw[fill] \point circle (0.07);

\draw[-,color=red] (-4,0) -- (-2,0.2) -- (0,0.9) -- (2,1.3) -- (4,2.4) -- (6,2.8) -- (8,3.4) -- (10,3.4) -- (12,4.1) -- (14,2.5) -- (16,1.8) -- (18,0);            
\foreach \point in {(-4,0), (-2,0.2), (0,0.9), (2,1.3), (4,2.4), (6,2.8), (8,3.4), (10,3.4), (12,4.1), (14,2.5), (16,1.8), (18,0)}
   \draw[fill, color=red] \point circle (0.07);

 \draw[line width=1.0pt,decorate,decoration={brace,amplitude=8pt}] (18.0,-3.5) -- (10.1,-3.5) node [black,midway,yshift=-0.5cm] {\footnotesize $r\cdot \ell_r$};
 \draw[line width=1.0pt,decorate,decoration={brace,amplitude=8pt}] (9.9,-3.5) -- (-4.,-3.5) node [black,midway,yshift=-0.5cm] {\footnotesize $d$};
 \draw[line width=1.0pt,decorate,decoration={brace,amplitude=8pt}] (-4,7.5) -- (18.,7.5) node [black,midway,yshift=0.5cm] {\footnotesize $\ba_B\rel \widetilde \bb_B^{*}$};

\foreach \y in {-2,6}
 \draw[-] (-4,\y) -- (18,\y);

\foreach \x in {-2,...,9}
   \draw[fill] (2*\x,-2) circle (0.07);

\foreach \x in {-2,...,9}
   \draw[fill] (2*\x,6) circle (0.07);

\end{tikzpicture}
\caption{Two representatives for Case II. Numbers of discretization points are linked by $ d+r \ell_r = d_0+q$.}
\label{fig:7C2}
\end{figure}


{\it Case III:} 
 By Remark \ref{rem:extendedMI} we
use a(n) (extended) chained Moser isotopy, denoted by $\widetilde\phi_t$, with $r=0$ and $\rho=1$ and $\ell_\rho = 2\lambda+2\ge 4$.
We denote the associated discrete braid by $\widetilde \bb_B:= \bigl\{ \pi_x \widetilde\phi_{t_j}(B)\bigr\}$.
 By construction $\widetilde \bb_B \possim \T^{2}\bb_B \possim \T^{2\lambda+2}\bb$.
The integers $q\ge 0$ and  $\kappa$ are chosen large enough such that $\T^{2\lambda+2}(\E^q\bb^*)\sim \widetilde \bb_B^*$ and $[\T^{2\lambda+2}\E^{q}(\ba\rel\bb^*)]$ is free, and where  $\bb_B^* = \bigl\{ \pi_x \widetilde \phi_{t_j}(B^*)\bigr\}$.
Define $\widetilde \bb_B^{*\#} = \bigl\{ \pi_x \widetilde \phi_{t_j}(B^{*\#})\bigr\}$, with $B^{*\#} = B\cup \{-4,-2,2,4\}$, which is equivalent to
augmentating $\widetilde \bb_B^*$ with the strands $\bb_B^s = \{(-1)^{j+1}4\}_j$ and $\bb_B^n = \{(-1)^j 4\}_j$.
Choose $\ba_B \rel \widetilde \bb_B^{*} \sim \T^{2\lambda+2}\E^q(\ba\rel\bb^*)$.
From \eqref{eqn:thedual} and the fact that $\E^q(\ba\rel\bb) = \E^q\ba\rel\E^q\bb$ and 
$\T^{2\lambda+2}\bigl(\E^q\ba\rel (\E^q\bb^*)^\#\bigr) \sim \ba_B \rel \widetilde \bb_B^{*\#}$, cf.\ Figure \ref{fig:7C3},
we derive that
\[
\begin{aligned}
0\not = \HHH(\ba &~\rel\bb^*) \wedge  \sbb^{2n(\lambda+1)} = \\
& = \HHH\bigl(\E^q(\ba\rel\bb^*)\bigr) \wedge \sbb^{2n(\lambda+1)}  
= \HHH\bigl(\E^q\ba\rel(\E^q\bb^*)^\#\bigr) \wedge \sbb^{2n(\lambda+1)}\\
&= \HHH\bigl(\T^{2\lambda+2}\bigl(\E^q\ba\rel (\E^q\bb^*)^\#\bigr)\bigr)  =\HHH\bigl(\ba_B \rel \widetilde \bb_B^{*\#} \bigr),
\end{aligned}
\]
which implies that $\HHH\bigl(\ba_B \rel \widetilde \bb_B^{*\#} \bigr) \not = 0$.
As in the previous case the non-triviality of the  braid invariant proves the theorem in Case III
since the braid class $[\ba_B \rel \widetilde \bb_B^{*\#}]$ is free, bounded and proper and stationary braids of $W_d$ in $[\ba_B \rel \widetilde \bb_B^{*\#}]$ 
yield invariant sets $A\subset \inter \disc$ by Proposition \ref{prop:reallyindisc}.
\end{proof}

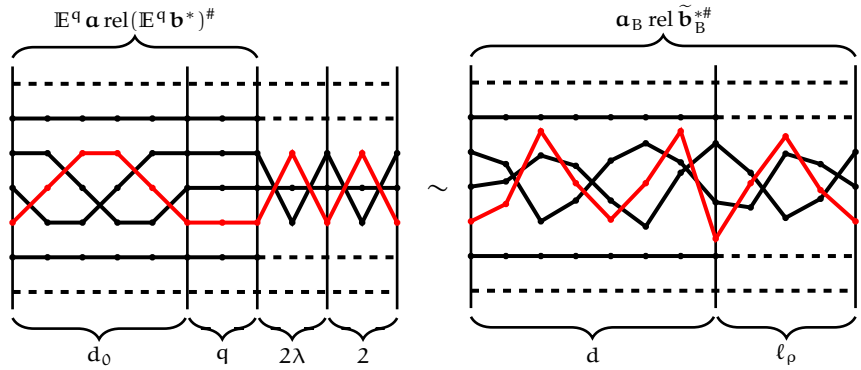
\begin{figure}[hbt]
\centering
\begin{tikzpicture}[xscale=0.23, yscale=0.23, line width = 1.5pt]

\foreach \x in {-2,8,12,16,20}
 \draw[line width=1.0pt, -] (\x,-5) -- (\x,9); 

\draw[-] (-2,2) -- (0,0) -- (2,0) -- (4,2) -- (6,4) -- (8,4) -- (10,4) -- (12,4) -- (14,0) -- (16,4) -- (18,0) -- (20,4);            
\foreach \point in {(-2,2), (0,0), (2,0), (4,2), (6,4), (8,4), (10,4), (12,4), (14,0), (16,4), (18,0), (20,4)}
   \draw[fill] \point circle (0.07);

\draw[-] (-2,4) -- (0,4) -- (2,2) -- (4,0) -- (6,0) -- (8,2) -- (10,2) -- (12,2) -- (14,2) -- (16,2) -- (18,2) -- (20,2);            
\foreach \point in {(-2,4), (0,4), (2,2), (4,0), (6,0), (8,2), (10,2), (12,2), (14,2), (16,2), (18,2), (20,2)}
   \draw[fill] \point circle (0.07);

\draw[-,color=red] (-2,0) -- (0,2) -- (2,4) -- (4,4) -- (6,2) -- (8,0) -- (10,0) -- (12,0) -- (14,4) -- (16,0) -- (18,4) -- (20,0);            
\foreach \point in {(-2,0), (0,2), (2,4), (4,4), (6,2), (8,0), (10,0), (12,0), (14,4), (16,0), (18,4), (20,0)}
   \draw[fill, color=red] \point circle (0.07);

 \draw[line width=1.0pt,decorate,decoration={brace,amplitude=8pt}] (7.9,-5.5) -- (-2.,-5.5) node [black,midway,yshift=-0.5cm] {\footnotesize $d_0$};
 \draw[line width=1.0pt,decorate,decoration={brace,amplitude=8pt}] (11.9,-5.5) -- (8.1,-5.5) node [black,midway,yshift=-0.5cm] {\footnotesize $q$};
 \draw[line width=1.0pt,decorate,decoration={brace,amplitude=8pt}] (15.9,-5.5) -- (12.1,-5.5) node [black,midway,yshift=-0.5cm] {\footnotesize $2\lambda$};
 \draw[line width=1.0pt,decorate,decoration={brace,amplitude=8pt}] (20,-5.5) -- (16.1,-5.5) node [black,midway,yshift=-0.5cm] {\footnotesize $2$};
 \draw[line width=1.0pt,decorate,decoration={brace,amplitude=8pt}] (-2,9.5) -- (12.,9.5) node [black,midway,yshift=0.5cm] {\footnotesize $\E^q\ba\rel(\E^q\bb^*)^\#$};
   
\foreach \y in {-2,6}
 \draw[-] (-2,\y) -- (12,\y);

\foreach \x in {-1,...,6}
   \draw[fill] (2*\x,-2) circle (0.07);

\foreach \x in {-1,...,6}
   \draw[fill] (2*\x,6) circle (0.07);

\foreach \y in {-2,6}
 \draw[-, dashed] (12,\y) -- (20,\y);

\foreach \y in {-4,8}
 \draw[-, dashed] (-2,\y) -- (20,\y);

\end{tikzpicture}
\raisebox{67pt}{
\begin{minipage}{0.05\linewidth}
  \centering \Large $\sim$
\end{minipage}}
\begin{tikzpicture}[xscale=0.23, yscale=0.23, line width = 1.5pt]

\foreach \x in {-2,12,20}
 \draw[line width=1.0pt, -] (\x,-5) -- (\x,9); 

\draw[-] (-2,2) -- (0,2.3) -- (2,3.8) -- (4,3.2) -- (6,1.2) -- (8,-0.3) -- (10,2.8) -- (12,4.5) -- (14,2.8) -- (16,0.2) -- (18, 1.3) -- (20,4);            
\foreach \point in {(-2,2), (0,2.3), (2,3.8), (4,3.2), (6,1.2), (8,-0.3), (10,2.8), (12,4.5), (14,2.8), (16,0.2), (18, 1.3), (20,4)}
   \draw[fill] \point circle (0.07);
  
\draw[-] (-2,4) -- (0,3.3) -- (2,0.0) -- (4,1.2) -- (6,3.5) -- (8,4.5) -- (10,3.4) -- (12,1.1) -- (14,0.8) -- (16,3.9) -- (18, 3.3) -- (20, 2);            
\foreach \point in {(-2,4), (0,3.3), (2,0.0), (4,1.2), (6,3.5), (8,4.5), (10,3.4), (12,1.1), (14,0.8), (16,3.9), (18, 3.3), (20, 2)}
   \draw[fill] \point circle (0.07);

\draw[-,color=red] (-2,0) -- (0,1.0) -- (2,5.2) -- (4,2.2) -- (6,0.1) -- (8,2.2) -- (10,5.2) -- (12,-1.0) -- (14,2.2) -- (16,4.9) -- (18, 1.8) -- (20,0);            
\foreach \point in {(-2,0), (0,1.0), (2,5.2), (4,2.2), (6,0.1), (8,2.2), (10,5.2), (12,-1.0), (14,2.2), (16,4.9), (18, 1.8), (20,0)}
   \draw[fill, color=red] \point circle (0.07);

 \draw[line width=1.0pt,decorate,decoration={brace,amplitude=8pt}] (11.9,-5.5) -- (-2.,-5.5) node [black,midway,yshift=-0.5cm] {\footnotesize $d$};
 \draw[line width=1.0pt,decorate,decoration={brace,amplitude=8pt}] (20.0,-5.5) -- (12.1,-5.5) node [black,midway,yshift=-0.5cm] {\footnotesize $\ell_\rho$};
 \draw[line width=1.0pt,decorate,decoration={brace,amplitude=8pt}] (-2,9.5) -- (20.,9.5) node [black,midway,yshift=0.5cm] {\footnotesize $\ba_B\rel \widetilde \bb_B^{*\#}$};
   
\foreach \y in {-2,6}
 \draw[-] (-2,\y) -- (12,\y);

\foreach \x in {-1,...,6}
   \draw[fill] (2*\x,-2) circle (0.07);

\foreach \x in {-1,...,6}
   \draw[fill] (2*\x,6) circle (0.07);

\foreach \y in {-2,6}
 \draw[-, dashed] (12,\y) -- (20,\y);

\foreach \y in {-4,8}
 \draw[-, dashed] (-2,\y) -- (20,\y);

\end{tikzpicture}
\caption{Two representatives for Case III. Numbers of discretization points are linked by $ d+ \ell_\rho = d_0+q+2\lambda +2$.
The dashed lines indicate the oscillating strands due to  the action of $\T$ and the augmentation of the skeleta by $^\#$. }
\label{fig:7C3}
\end{figure}

\begin{exm}
\label{exm:exist2}
Consider $F\in \Symp(\disc)$ and assume that $F$ has a four point invariant sets $B$ and  mapping class $[F]$ relative to
$B$ is given by $\beta = \sigma_3\sigma_1\sigma_2^2\sigma_1\sigma_3$, cf.\ Figure \ref{fig:braidex1}.
Consider the 2-colored braid class $\llb\gamma,\aset\rrb \in \CC\BB_{1,4}^+$, with
\[
\gamma = \sigma_4\sigma_1\sigma_2\sigma_3\sigma_2^2\sigma_3\sigma_2\sigma_1\sigma_4,\quad \aset = \{3\}.
\]
Then, $(\gamma,\aset) \mapsto \beta$ and the positive conjugacy class is proper and acylindrical.
%
In \cite[Sect.\ 4.5]{BraidConleyIndex} the braid invariant is given: $\HHH\llb\gamma,\aset\rrb = \sbb^1$, which by Theorem \ref{thm:main1}
yields another fixed point.
The index calculations in \cite{BraidConleyIndex} suggest that we can find more periodic point for $F$. Instead of using multi-strand braids with $n>1$, we
use $F^k$ instead.
Consider three different 2-colored braid words: $\gamma_0 = \gamma$ as above, $\gamma_{-1} = 
 \sigma_4\sigma_1\sigma_2\sigma_3^2\sigma_2\sigma_1\sigma_4$, and
$\gamma_1  = \sigma_4\sigma_1\sigma_3\sigma_2^2\sigma_3\sigma_1\sigma_4$.
For all three cases the skeletal word is given by $\beta$, the coloring is given by $\aset=\{3\}$.
Consider a symbolic sequence $\{a_i\}_{i=1}^k$, $a_i\in \{-1,0,1\}$, then
the positive conjugacy class $(\gamma,\aset)$, with $\aset = \{3\}$, and
\[
\gamma = \gamma_{a_0} \cdot\gamma_{a_1} \cdots \gamma_{a_{k-1}}\cdot \gamma_{a_k},
\]
is proper and acylindrical, except for $a_i=-1$, or $a_i=1$ for all $i$.
If follows that $(\gamma,\aset) \mapsto \beta^k$. In \cite[Sect.\ 4.5]{BraidConleyIndex} the braid invariant is given by $\HHH\llb\gamma,\aset\rrb = \sbb^r$,
where $r$ is the number of zeroes in $\{a_i\}$.
This procedure produces many $k$-periodic point for $F$ that are forced by the invariant set $B$. As matter of fact one obtains a lower bound on the 
topological entropy of $F$.
\end{exm}

\begin{exm}
\label{exm:exist3}
%
%
Let $F\in \Symp(\disc)$ possess an invariant set $B$ consisting of three points,
and let the mapping class of $[F]$ relative to
$B$ be represented by a positive braid word $\beta = \ii_m^{-1}[\bbeta]$,
with $\bbeta = \bigl\{\bbeta^1(t), \bbeta^2(t), \bbeta^3(t)\bigr\}$ being a geometric representative of 
\[
\beta = \sigma_1\sigma_2^2\sigma_1^2\sigma_2^2\sigma_1\sigma_2\sigma_1\sigma_2^2\sigma_1,
\]
cf.\ Figure \ref{ex:nontriv1}.
For the intersection numbers it holds that  $\iota\bigl(\bbeta^1(t),\bbeta^2(t)\bigr) = \iota\bigl(\bbeta^1(t),\bbeta^3(t)\bigr) = 6$ and
  $\iota\bigl(\bbeta^2(t),\bbeta^3(t)\bigr) = 1 < 6$.
From considerations in~\cite[Sect.\ 9.2]{BraidConleyIndex} it follows that through the black strands $\bbeta$ one can plot a single red strand $\balpha$, 
such that the 2-colored braid class 
$\llb\gamma,\aset\rrb$ represented by their union 
\[
\gamma = \sigma_2 \sigma_1 \sigma_2 \sigma_3 \sigma_2 \sigma_3 \sigma_2 \sigma_3 \sigma_1
\sigma_2 \sigma_2 \sigma_1 \sigma_2 \sigma_3 \sigma_2 \sigma_3 \sigma_2 \sigma_1 \sigma_2 \sigma_1 \sigma_2 \sigma_3 \sigma_2 \sigma_3 \sigma_2 \sigma_1 \sigma_2,
\]
with $\aset = \{2\}$, cf.\ Figure \ref{ex:nontriv1}, 
is proper, acylindrical and of nontrivial index.
The intersection numbers for $\balpha$ are
 $\iota\bigl(\balpha(t),\bbeta^2(t)\bigr) = \iota\bigl(\balpha(t),\bbeta^3(t)\bigr) = 4$.
If the intersection numbers are chosen more generally, i.e.\
$\iota\bigl(\bbeta^1(t),\bbeta^2(t)\bigr) = \iota\bigl(\bbeta^3(t),\bbeta^3(t)\bigr) = 2p$ and
  $\iota\bigl(\bbeta^2(t),\bbeta^3(t)\bigr) = r < 2p$, and  $\iota\bigl(\balpha(t),\bbeta^2(t)\bigr) = \iota\bigl(\balpha(t),\bbeta^3(t)\bigr) = 2q$,
where $r > 0$ and $p \geq 2$. If $r < 2q < 2p$, then 
the singular homology of $\llb\gamma,\aset\rrb$ is given by
$$
H_{k}\bigl(\HHH\llb\gamma,\aset\rrb\bigr) = \begin{cases}
  \mathbb{R}: k= 2q,\ 2q+1,\\
 0: \text{elsewise}.\\
\end{cases}
$$
By Theorem~\ref{thm:main1} we conclude that there are at least two additional distinct fixed points $A_1, A_2$
with $\jmath_{A_i\cup B}([F]) = \gamma\!\!\mod \square, \ i = 1,2$. 
In addition,
via concatenating braid diagrams one can produce an infinite number of periodic solutions of different periods, cf.\ \cite[Lemma 47]{BraidConleyIndex}.
The above 
forcing result
is specific for area-preserving mapping of the 2-disc, or $\rr^2$ and \emph{not} true for arbitrary diffeomorphisms of the 2-disc.

For example consider the time-1 mapping $F\colon \disc\to\disc$ given by the differential equation $\dot{r} = r(r-a_1)(r-a_2)(r-1)$ and $\dot{\theta} = g(r)>0$,
with $g(a_1) = \pi$ and $g(a_2)  = 6\pi$, and $0<a_1<a_2<1$.
The set $B = \{-a_1,a_1,a_2\}$ is invariant and the mapping class of $F$ relative to $B$ is given by
the skeleton $\bbeta$. The mapping $F$ has no invariant set matching the braid $\gamma$.
This implies that the invariants in \cite{JiangZheng} are void for this example and the braid invariant introduced in this
paper add additional information to existing invariants,   in particular in the area-preserving case.
\end{exm}

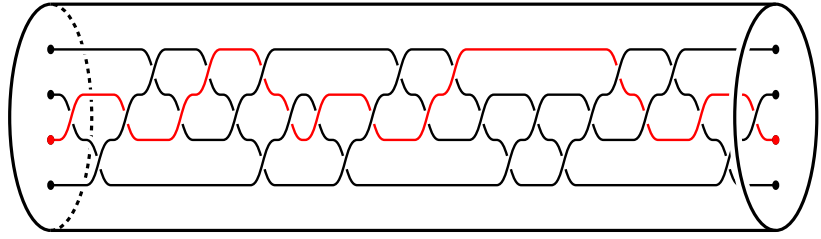
\begin{figure}[hbt]
  \centering
\scalebox{0.6}{                
\begin{tikzpicture}[rotate=90, line width=2.0pt, yscale=0.6] 
\draw[dashed] (0,0) arc (180:360:2.5 and 1.5);
\draw (5,0) arc (0:180:2.5 and 1.5);

  \braid[color=white,
 line width=6.5pt,
] s_2 s_1 s_2  s_3 s_2 s_3  s_2 s_3-s_1 s_2  s_2 s_1 s_2  s_3 s_2 s_3  s_2 s_1  s_2 s_1 s_2  s_3 s_2 s_3  s_2 s_1 s_2;

  \braid[
 line width=1.5pt,
 style strands={2}{red}
] s_2 s_1 s_2  s_3 s_2 s_3  s_2 s_3-s_1 s_2  s_2 s_1 s_2  s_3 s_2 s_3  s_2 s_1  s_2 s_1 s_2  s_3 s_2 s_3  s_2 s_1 s_2;

\draw[color=white, line width=8.0pt] (5,-26.5) arc (0:180:2.5 and 1.5);
 
\draw (2.5,-26.5) ellipse (2.5cm and 1.5cm);

\draw[-] (5,0) -- (5,-26.5); 
\draw[-] (0,0) -- (0,-26.5); 

\foreach \x in {1,...,4}
 \draw[fill] (\x,0) circle (0.07);
\foreach \x in {1,...,4}
 \draw[fill] (\x,-26.5) circle (0.07);

\draw[color=red,fill] (2,0) circle (0.07);
\draw[color=red,fill] (2,-26.5) circle (0.07);
\end{tikzpicture}
}
\caption{A geometric representative of the braid for Example~\ref{exm:exist3} with $p=3$, $q=2$, and $r=1$.}
\label{ex:nontriv1}
\end{figure}


\begin{rem}
One can also consider 2-colored words $(\gamma,\aset)$, $\gamma\in \BB_{n+m}$. Braid words can be expressed in normal form.
The fundamental element, or Garside element in $\BB_m^+$ is denoted by 
$
\triangle := (\sigma_{1}\cdots \sigma_{m-1}) (\sigma_{1}\cdots \sigma_{m-2})\cdots  (\sigma_1 \sigma_2)\sigma_1,
$
and is also referred to as a \emph{half twist}.
The element $\triangle$ is a factor of a positive word $\beta$, if $\beta = \beta' \triangle \beta''$, $\beta',\beta''$ positive (possibly trivial) words.
If not, $\beta$ is said to be \emph{prime} to $\triangle$.
If we use the lexicographical order on the presentations of a positive word $\beta$, then the smallest positive word $\beta'$, positively equal to $\beta$, is called the \emph{base}
of $\beta$. For positive words $\beta$, prime to $\triangle$, the base is denoted by $\bar\beta$.
Following \cite{Birman} and \cite{Hazewinkel4},
every word $\beta\in \BB_m$ is uniquely presented by a word $\triangle^\varrho  \bar\delta$, with $\varrho\in \zz$ and $\bar\delta\in \BB_m^+$,
which is called \emph{left Garside normal form}.
Via the relation $\triangle \beta = r(\beta)\triangle$ we obtain the \emph{right Garside normal form} $\beta = r(\bar\delta)\triangle^\varrho$.

For the 2-colored braid words the Garside normal form is not the appropriate normal form, since  
 odd powers of $\triangle$ do not represent trivial permutations.
In the case that the Garside power of $\gamma$ is even, then $\triangle^2 = \square$ defines the identity permutation and therefore the associated base $(\bar\delta,\aset)$, with $\bar\delta \in \BB_{n+m}^+$, is a positive 2-colored braid word. In this case the left and right normal form are the same since $\square$ is at the center of the braid group
$\BB_{n+m}$.
When the Garside power  is odd, we argue as follows. Let $\varrho=2\uplambda +1$, then
\begin{equation}
\label{eqn:garside2}
\gamma=\triangle^{2\uplambda +1} \bar \delta = \square^\uplambda \triangle \bar \delta
= \square^\uplambda r(\bar \delta) \triangle =   r(\bar \delta) \triangle \square^\uplambda =   \bar \epsilon \square^\uplambda
= \square^\uplambda\bar \epsilon,
\end{equation}
where $\bar\epsilon = \triangle \bar\delta = r(\bar\delta)\triangle$ is the base and $\bar\epsilon$ is prime to $\square$.
The power $\uplambda$ is related to the Garside power via $\uplambda = \uplambda(\gamma) = \floor{\varrho/2}$ and is referred to as the \emph{symmetric Garside power}
of $\gamma$.
For 2-colored braid words  $(\gamma,\aset)$ this yields the following symmetric normal form: $\gamma = \square^\uplambda  \bar\epsilon = \bar\epsilon \square^\uplambda$, with $\uplambda \in \zz$ and $(\bar\epsilon,\aset)$ a positive 2-colored braid word.
Since mapping classes of $F$ are given modulo full twists  the latter normal form suggests that we capture all forcing via positive braid words. 
If conjugacy is incorporated, the power $\uplambda$ can be optimized with respect to different conjugate braid words.
\end{rem}

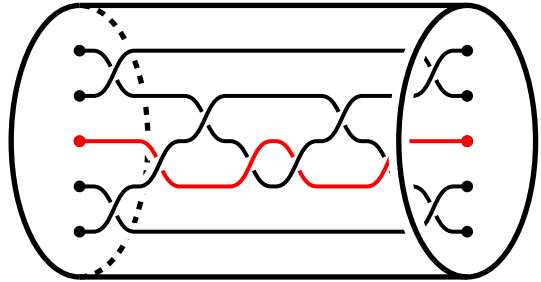
\begin{figure}[hbt]
\centering
\begin{tikzpicture}[rotate=90, line width=2.0pt, scale=0.6]
\draw[dashed] (0,0) arc (180:360:3.0 and 1.5);
\draw (6,0) arc (0:180:3.0 and 1.5);

  \braid[color=white,
 line width=6.5pt
] s_4-s_1 s_2 s_3 s_2 s_2 s_3 s_2 s_1-s_4;

  \braid[
 line width=1.5pt,
 style strands={3}{red}
] s_4-s_1 s_2 s_3 s_2 s_2 s_3 s_2 s_1-s_4;

\draw[color=white, line width=8.0pt] (6,-8.5) arc (0:180:3.0 and 1.5);
 
\draw (3,-8.5) ellipse (3cm and 1.5cm);

\draw[-] (6,0) -- (6,-8.5); 
\draw[-] (0,0) -- (0,-8.5); 

\foreach \x in {1,...,5}
 \draw[fill] (\x,0) circle (0.07);
\foreach \x in {1,...,5}
 \draw[fill] (\x,-8.5) circle (0.07);

\draw[color=red,fill] (3,0) circle (0.07);
\draw[color=red,fill] (3,-8.5) circle (0.07);
\end{tikzpicture}
\caption{The canonical representation of the  2-colored braid class $\llb\gamma,\aset\rrb \in \CC\BB_{1,4}^+$
with $\gamma = \sigma_4\sigma_1\sigma_2\sigma_3\sigma_2^2\sigma_3\sigma_2\sigma_1\sigma_4 $ and $\aset = \{3\}$.}
\label{fig:braidex1}
\end{figure}

\appendix

\section{Mapping classes and braids}
\label{sec:MCGclbr}

We give overview various known and less known facts about mapping class groups of the 2-disc and the 2-disc with  a finite number of marked points.

\subsection{Mapping class groups  of the 2-disc}
By the   `Alexander trick', cf.\ \cite[pp.\ 48]{Margalit}, $\Homeo^+_0(\disc)$ is contractible and thus $\Mod_0(\disc)$ is trivial.
The case of $\Homeo^+(\disc)$ is more involved. Consider the Serre fibration $\Homeo^+(\disc) \to \Homeo^+(\partial \disc)$ with fiber $\Homeo^+_0(\disc)$, cf.\ \cite[pp.\ 536]{Sher} and the associated homotopy long exact sequence
\[
\begin{aligned}
\cdots \longrightarrow &~\pi_1\bigl(\Homeo^+_0(\disc)\bigr)  \longrightarrow \pi_1\bigl(\Homeo^+(\disc)\bigr) \longrightarrow\\ 
 \longrightarrow&~\pi_1\bigl(\Homeo^+(\partial\disc)\bigr) \longrightarrow \pi_0\bigl(\Homeo^+_0(\disc)\bigr)  \longrightarrow \\
 \longrightarrow&~\pi_0\bigl(\Homeo^+(\disc)\bigr) \longrightarrow \pi_0\bigl(\Homeo^+(\partial\disc)\bigr).    
\end{aligned}
\]
Since $\pi_k\bigl(\Homeo^+_0(\disc)\bigr)\cong 1$, for all $k\ge 0$, and $\pi_1\bigl(\Homeo^+(\partial \disc)\bigr) \cong \zz$ and
$\pi_0\bigl(\Homeo^+(\partial \disc)\bigr) \cong 1$, we have the short exact sequences
\[
1 \longrightarrow \pi_1\bigl(\Homeo^+(\disc)\bigr)  \longrightarrow \zz \longrightarrow 1,
\]
which yields that $\pi_1\bigl(\Homeo^+(\disc)\bigr)\cong \zz$, and
\[
1 \longrightarrow \pi_0\bigl(\Homeo^+(\disc)\bigr) \longrightarrow 1,
\]
which shows that $\Mod(\disc) \cong 1$ and completes the proof of Proposition \ref{prop:MCG11}.

\begin{rem}
\label{rmk:smooth1}
Mapping classes can also be defined for orientation preserving diffeomorphisms of the 2-disc, denoted by $\Diff^+(\disc)$  and $\Diff^+_0(\disc)$. Due to a result by Smale \cite{Smale}, $\Diff^+_0(\disc)$ is contractible 
and therefore the Serre fibration $\Diff^+(\disc) \to \Diff^+(\partial \disc)$ with fiber $\Diff^+_0(\disc)$,
yields that 
$\pi_0\bigl(\Diff^+( \disc)\bigr) \cong 1$ and $ \pi_0\bigl(\Diff^+_0(\disc)\bigr) \cong 1$,
and thus
the above
 mapping class groups can also be defined via diffeomorphisms, cf.\ \cite{Boldsen}, \cite{GG} \cite{LaBach}.
\end{rem}

For a treatment of mapping class groups via symplectic and Hamiltonian diffeomorphisms see, Appendix \ref{sec:sympMCG}.

\subsection{Braids and mapping classes}
\label{subsec:braidMCG}
In order to determine the mapping class group in Proposition \ref{prop:MCG12} we 
consider the fiber bundle
\[
\Homeo^+_0(\disc\rel B) \xrightarrow{i} \Homeo^+_0(\disc ) \xrightarrow{\epsilon} \C_m\disc,
\]
 cf.\ \cite[pp.\ 245]{Margalit}.
The associated homotopy long exact  sequence, with base point $B$, 
yields
\[
\begin{aligned}
\cdots \longrightarrow&~\pi_1\bigl(\Homeo^+_0(\disc)\bigr)\xrightarrow{\epsilon_*} \pi_1\bigl(\C_m\disc\bigr) \xrightarrow{d_*}\\
\xrightarrow{d_*}&~ \pi_0\bigl(\Homeo^+_0(\disc\rel B)\bigr)  \xrightarrow{i_*} \pi_0\bigl(\Homeo^+_0(\disc)\bigr) \longrightarrow 1.
\end{aligned}
\]
Since $\pi_k\bigl(\Homeo^+_0(\disc)\bigr)\cong 1$, for all $k\ge 0$, we obtain the short exact sequence
\[
1\longrightarrow \pi_1\bigl(\C_m\disc\bigr) \xrightarrow{d_*} \pi_0\bigl(\Homeo^+_0(\disc\rel B)\bigr) \longrightarrow 1,
\]
which proves that $\pi_1\bigl(\C_m\disc\bigr) \xrightarrow{d_*} \pi_0\bigl(\Homeo^+_0(\disc\rel B)\bigr)$ is an isomorphism and 
$\Mod_0(\disc\rel B) \cong \BB_m$ via $\iota_B\circ d_*^{-1}$.
Since $\Diff^+_0(\disc)$ is contractible, the same arguments yield
$\pi_0\bigl(\Diff^+_0(\disc\rel B)\bigr)\cong \BB_m$, which implies that
diffeomorphisms define the same mapping class group.
For a detailed study of relations between braids and mapping class groups
the reader is also referred to a comprehensive treatise by Birman \cite{Birman,Birman2}. See also \cite{GG}, \cite{G3}.

For homeomorphisms of the 2-disc  preserving $\partial\disc$ setwise we consider the fiber bundle
\[
\Homeo^+(\disc\rel B) \xrightarrow{i} \Homeo^+(\disc) \xrightarrow{\epsilon} \C_m\disc,
\]
which implies the following homotopy long exact sequence 
\[
\begin{aligned}
\cdots \longrightarrow &~ \pi_1\bigl(\Homeo^+(\disc)\bigr) 
 \xrightarrow{\epsilon_*} \pi_1\bigl(\C_m\disc\bigr)\xrightarrow{d_*}\\
  \xrightarrow{d_*}&~\pi_0\bigl(\Homeo^+(\disc\rel B)\bigr)  \xrightarrow{i_*}
 \pi_0\bigl(\Homeo^+(\disc)\bigr) \longrightarrow 1,
\end{aligned}
\]
with base point $B$.
From exactness we deduce that $\BBB_m\disc/\ker d_* \cong {\rm im~} d_* = \ker i_*$. For the kernel we have $\ker d_* = Z(\BBB_m\disc)$, cf.\ \cite[Theorem 4.3]{Birman}.
Since,
 $\pi_0\bigl(\Homeo^+(\disc)\bigr)\cong 1$ we have that $\ker i_* = \Mod(\disc\rel B)$ and
 \[
 \longrightarrow \pi_1\bigl(\Homeo^+(\disc)\bigr) 
 \xrightarrow{\epsilon_*} \BBB_m\disc\xrightarrow{d_*} \Mod(\disc\rel B) \xrightarrow{i_*} 1,
 \]
which implies that
$\Mod(\disc\rel B) \cong \BB_{m}/Z(\BB_{m})$.

The same can be derived using diffeomorophism instead of homeomorphisms, i.e. $\pi_0\bigl(\Diff^+(\disc\rel B) \bigr) = \Mod(\disc\rel B)
\cong \BB_m/Z(\BB_m)$, cf.\ \cite{GG}.

\section{Symplectic mapping classes}
\label{sec:sympMCG}
 
Two symplectomorphisms $F,G\in \Symp(\disc)$ are symplectically isotopic if there exists an isotopy $\phi_t$, with $\phi_t\in\Symp(\disc)$ for all $t\in [0,1]$,\footnote{The condition $\phi_t\in\Symp(\disc)$ is equivalent to $\phi_t^*\omega = \omega$.}
such that $\phi_0=F$ and $\phi_1 = G$.
The equivalence classes in $\Symp(\disc)/\!\!\sim$ are called \emph{symplectic mapping classes}. 

\begin{prop}
\label{prop:mapclass1}
$\pi_0\bigl(\Symp(\disc)\bigr) = \Mod(\disc)\cong 1$.
\end{prop}

\begin{proof}
%
Since $\Mod(\disc)\cong 1$, two symplectomorphisms
$F,G \in \Symp(\disc)$ are isotopic in $\Diff^+(\disc)$.
Therefore,  $F^{-1}G$ is isotopic to the identity via an isotopy $\xi_t$.
The isotopy $\xi_t$ does not necessarily preserve $\omega$. Define $\omega_t := \xi_t^*\omega$, 
which is a based loop in $\Omega^2(\disc)$, since
$\omega_t =\omega$ at $t=0$ and $t=1$.
Observe that $\int_{\disc} \omega_t = \int_{\disc} \omega = \pi$ for all
$t\in [0,1]$.
Indeed, since $\xi_t$ is smooth 1-parameter family of diffeomorphisms it holds that
$\int_{\disc} \omega_t = \pi \deg(\xi_t) = \pi$, since $\deg(\xi_t) = 1$ for all $t\in [0,1]$.

Write   $\omega_t = a_t(x,y) dx\wedge dy$,
with $a_t(x,y)>0$ on $[0,1]\times \disc$ and $a_0=a_1 =1$.
In order to construct a symplectic isotopy we invoke Moser's stability argument, cf.\ \cite{Moser23}, \cite{MS1}, Sect. 3.2.
Consider  potential functions $\Phi_t\colon \disc \times [0,1] \to \rr$ and define the vector fields
$X_t={1\over a_t(x,y)}\nabla \Phi_t$,
with $\langle X_t(x),{\bf n} \rangle  =0$ for
$x \in \partial \disc$, ${\bf n}$ the outward pointing normal.
The boundary condition guarantees that  $\disc$ is invariant for the associated flow $\chi_t$.
 Furthermore, define 1-forms
$\theta_t  = -\iota_{X_t} \omega_t$.
In order to apply Moser stability we seek potentials $\Phi_t$ such that $\frac{d\omega_t}{dt}  = d\theta_t$, which is equivalent to
the Neumann problem
\begin{equation}
\label{eqn:Neumann}
-\Delta \Phi_t = \frac{da_t(x,y)}{dt}, \quad \partial_{\bf n} \Phi_t\bigr|_{\partial \disc} =0.
\end{equation}
Since the forms $\omega_t$ are cohomologous, Stokes' theorem implies that
\[
0 = \int_\disc \frac{d\omega_t}{dt} = \int_\disc -\Delta\Phi_t \omega = \oint_{\partial\disc} \theta_t =\oint_{\partial\disc}\partial_{\bf n} \Phi_t
\]
which shows that the Neumann problem is well-posed and has a unique solution (up to an additive constant), which depends smoothly on $t$.

By construction $\chi_t^* \omega_t = \omega$ and the desired symplectic isotopy is given by
 $\xi_t \circ \chi_t$.
 The symplectic isotopy $\phi_t :=F\circ \xi_t\circ \chi_t$ is an isotopy between $F$ and $G$, which proves that $F$ and $G$ are symplectically isotopic, and thus $\pi_0\bigl(\Symp(\disc)\bigr) = \Mod(\disc)$.
 \end{proof}

As for homeomorphisms we can also consider 
 relative symplectic mapping classes.
Two symplectomorphisms $F,G\in \Symp(\disc)$ are of the same relative symplectic mapping class if there
exists an isotopy $\phi_t$,  with $\phi_t\in\Symp(\disc)$ and $\phi_t(B) = B$ for all $t\in [0,1]$, such that
$\phi_0=F$ and $\phi_1=G$. The subgroup of such symplectomorphisms is denoted by $\Symp(\disc\rel B)$.


 \begin{lem}
\label{prop:mapclass2}
Let  $F,G\in \Symp(\disc)$ be isotopic in $\Diff^+(\disc\rel B)$, then
they are isotopic in $\Symp(\disc\rel B)$.
\end{lem}

\begin{proof}
Let $F,G\in \Symp(\disc)$ be isotopic in $\Diff^+(\disc\rel B)$,
 then  $F^{-1}G$ is isotopic to the identity via a smooth isotopy $\xi_t$ with the additional condition
that $\xi_t(B) = B$ for all $t\in [0,1]$.
In order to find a symplectic isotopy we
repeat the proof of Proposition \ref{prop:mapclass1} with a few modifications.

Let  $X_t = {1\over a_t(x,y)}\nabla \Phi_t+  {1\over a_t(x,y)} J\nabla \lambda_t(x)$,
where $\lambda_t\colon [0,1]\times \disc \to \rr$  is a smooth function compactly supported in $\inter \disc$,
 where ${\bf t} = J{\bf n}$ is a unit tangent such that $\{{\bf n}, {\bf t}\}$ is positively oriented.
In order for the associated flow $\chi_t$ to restrict to a flow on $\disc$ we need that 
${\partial \lambda_t\over \partial {\bf t}} +
 {\partial \Phi_t\over \partial {\bf n}}=0$ at $\partial \disc$.
Since $\lambda_t$ is compactly supported in $\inter\disc$ we use the Neumann boundary condition for $\Phi_t$.
For the 
 1-forms  we obtain
\[
\theta_t = -\iota_{X_t}\omega_t =  -\partial_x \Phi_t dy + \partial_y \Phi_t dx - d\lambda_t.
\]
As in the proof of Proposition \ref{prop:mapclass1} the potential $\Phi_t$ is determined by the Neumann problem in \eqref{eqn:Neumann}.
Since $\lambda_t$ can be chosen arbitrarily we define
$\lambda_t(x) = \langle J\nabla \Phi_t(z_0),x\rangle$ on  neighborhoods of $z_0\in B$ and a smooth extension
outside, compactly supported in $\inter\disc$.
This guarantees that the   isotopies $\chi_t$ preserve $B$ point wise, which completes the proof.
\end{proof}

Let $\phi_t$ be a path in $\Symp(\disc)$, then $\phi_t$ satisfies the initial value
problem ${d\over dt} \phi_t = X_t \circ \phi_t$, where $X_t = {d\over dt} \phi_t \circ \phi_t^{-1}$ is a time-dependent vector field.
Since $\phi_t^*\omega = \omega$ it holds that $d\iota_{X_t}\omega=0$ for all $t\in [0,1]$, cf.\ \cite{MS1}, Proposition 3.2.
A symplectic isotopy is Hamiltonian is there exists a smooth function $H\colon[0,1]\times \disc \to \rr$
such that $\iota_{X_t}\omega = -dH(t,\cdot)$.
In this case $F=\phi_1$ is   a Hamiltonian symplectomorphism. 

By construction $\langle X_t,{\bf n}\rangle =0$ and
\[
-dH(t,\cdot)({\bf t}) = \iota_{X_t}\omega({\bf t}) = \omega(X_t,{\bf t}) = \omega(X_t,J{\bf n}) = \langle X_t,{\bf n}) =0,
\]
which proves that $H(t,\cdot)|_{\partial \disc} =  const.$
Since the 2-disc $\disc$ is contractible such a Hamiltonian exists, showing that symplectomorphisms of the 2-disc are Hamiltonian.

Let $\Symp(\disc\rel B)$ be the symplectomorphisms which leave a set $B \subset \inter \disc$ invariant.
The subgroup of Hamiltonian symplectomorphisms is denoted by $\Ham(\disc\rel B)$.
The above procedure yields the following result.

 \begin{prop}
\label{prop:mapclass3}
$\Symp(\disc\rel B) = \Ham(\disc\rel B)$.
\end{prop}

\begin{proof}
From the previous it follows that there exists a Hamiltonian such that $\iota_{X_t}\omega = -dH(t,\cdot)$.
Since the points in $B$ are rest points for $X_t$ it follows that $X_t(z_0) = 0$ for all $t$ and for all $z_0\in B$.
Consequently, the points in $B$ are critical points of $H$.
\end{proof}


\renewcommand{\refname}{\spacedlowsmallcaps{References}} 

\bibliographystyle{unsrt}

\bibliography{references.bib} 

\end{sloppypar}
\end{document}

%% file: structure.tex
\usepackage[
nochapters, 
beramono, 
eulermath,
pdfspacing, 
dottedtoc 
]{classicthesis} 

\usepackage{arsclassica} 

\usepackage[T1]{fontenc} 

\usepackage[utf8]{inputenc} 

\usepackage{graphicx} 
\graphicspath{{Figures/}} 

\usepackage{enumitem} 

\usepackage{lipsum} 

\usepackage{subfig} 

\usepackage{amsmath,amssymb,amsthm} 

\usepackage{varioref} 
\usepackage{geometry}

\theoremstyle{definition} 

\theoremstyle{plain} 

\theoremstyle{remark} 

\hypersetup{
colorlinks=true, breaklinks=true, bookmarks=true,bookmarksnumbered,
urlcolor=webbrown, linkcolor=RoyalBlue, citecolor=webgreen, 
pdftitle={}, 
pdfauthor={\textcopyright}, 
pdfsubject={}, 
pdfkeywords={}, 
pdfcreator={pdfLaTeX}, 
pdfproducer={LaTeX with hyperref and ClassicThesis} 
}